%% file: main.tex
\documentclass[11pt]{article}
\input{preamble}

\title{Joint Learning of Linear Dynamical Systems under Smoothness Constraints}

\author{Hemant Tyagi\footnote{The first version of this paper was prepared when the author was affiliated to Inria Lille. This work was supported by a Nanyang Associate Professorship (NAP) grant from NTU Singapore.} \\ 
  Division of Mathematical Sciences, \\ School of Physical and Mathematical Sciences, \\ 
  Nanyang Technological University Singapore \\ 
Email: \texttt{hemant.tyagi@ntu.edu.sg}}

\begin{document}
\maketitle

\begin{abstract}
We consider the problem of joint learning of multiple linear dynamical systems. This has received significant attention recently under different types of assumptions on the model parameters. The setting we consider involves a collection of $m$ linear \rev{systems, each} of which resides on a node of a given undirected graph $G = ([m], \calE)$. We assume that the system matrices are marginally stable, and satisfy a smoothness constraint w.r.t $G$ -- the smoothness measured in a manner akin to the quadratic variation of a signal on a graph. Given access to the states of the nodes over $T$ time points, we then propose two estimators for joint estimation of the system matrices, along with non-asymptotic error bounds on the mean-squared error (MSE). In particular, we show conditions under which  the MSE converges to zero as $m$ increases, typically polynomially fast w.r.t $m$. The results hold under mild (i.e., $T \sim \log m$), or sometimes, even no assumption on $T$ (i.e. $T \geq 2$). 
\end{abstract}

\input{intro}

\input{prob_setup}

\input{main_results}

\input{analysis_laplacian_smoothing}

\input{analysis_subspace_constraint}

\input{exps}

\input{concluding_remarks}

\appendix

\input{appendix}

\newpage 

\bibliographystyle{abbrvnat}
\bibliography{references}

\end{document}

%% file: preamble.tex
\usepackage[margin=1in]{geometry}
 
\usepackage[]{amsmath,amssymb,epsfig}
\usepackage{amsthm}
\usepackage{amsmath,mathtools}
\usepackage{amssymb}
\usepackage{graphicx}
\usepackage{epstopdf}
\usepackage{comment}
\usepackage{array}
\usepackage{algorithm}
\usepackage{url}

\usepackage{lscape}
\usepackage{algpseudocode}
\usepackage{setspace}
\usepackage{multicol}
\usepackage{multirow}
\usepackage{color}
\usepackage{colortbl}
\usepackage{xcolor}
\usepackage{hyperref}

\usepackage{placeins}

\usepackage[%
    font={small,sf},
    labelfont=bf,
    format=hang,    
    format=plain,
    margin=0pt,
    width=0.8\textwidth,
]{caption}
\usepackage[list=true]{subcaption}
\usepackage[round]{natbib}

\newtheorem{theorem}{Theorem}

\newtheorem{assumption}{Assumption}

\newtheorem{corollary}{Corollary}

\newtheorem{definition}{Definition}
\newtheorem{example}{Example}

\newtheorem{lemma}{Lemma}

\newtheorem{proposition}{Proposition}
\newtheorem{remark}{Remark}

\numberwithin{equation}{section}

\DeclareMathOperator{\Tr}{Tr}

\DeclareMathOperator{\diag}{diag}
\DeclareMathOperator{\blkdiag}{blkdiag}

\DeclareMathOperator{\vect}{vec}

\DeclareMathOperator{\poly}{poly}

\newcommand{\calB}{\ensuremath{\mathcal{B}}}

\newcommand{\calF}{\ensuremath{\mathcal{F}}}

\newcommand{\calN}{\ensuremath{\mathcal{N}}}

\newcommand{\calK}{\ensuremath{\mathcal{K}}}

\newcommand{\calE}{\ensuremath{\mathcal{E}}}

\newcommand{\Mtil}{\ensuremath{\tilde{M}}}

\newcommand{\norm}[1]{\big\|{#1}\big\|}
\newcommand{\abs}[1]{\left|{#1}\right|}

\newcommand{\set}[1]{\left\{{#1}\right\}}
\newcommand{\dotprod}[2]{\Big\langle#1,#2\Big\rangle}
\newcommand{\est}[1]{\widehat{#1}}
\newcommand{\expec}{\ensuremath{\mathbb{E}}}
\newcommand{\matR}{\ensuremath{\mathbb{R}}}

\newcommand{\argmin}[1]{\underset{#1}{\operatorname{argmin}}}

\newcommand{\prob}{\ensuremath{\mathbb{P}}}

\newcommand{\rev}[1]{\textcolor{black}{#1}}

\newcommand{\indic}{\ensuremath{\mathbf{1}}} 

\newcommand{\Xtil}{\widetilde{X}}
\newcommand{\Ubar}{\overline{U}}

\newcommand{\Mbar}{\overline{M}}

\newcommand{\xtil}{\widetilde{x}}
\newcommand{\Pproj}{P^{(\tau)}}
\newcommand{\Ptilproj}{\widetilde{P}^{(\tau)}}
\newcommand{\Mlam}{M_{\lambda}}
\newcommand{\Mlamtil}{\tilde{M_{\lambda}}}
\newcommand{\Mtau}{M_{\tau}}
\newcommand{\Vtau}{V_{\tau}}
\newcommand{\stoptime}{\overline{\tau}}
\newcommand{\lambmin}{\lambda_{\min}}

%% file: intro.tex
%
\section{Introduction} \label{sec:intro}
Learning a linear dynamical system (LDS) from samples of its trajectory is a fundamental problem with many applications, e.g., in control systems, reinforcement learning and time-series analysis to name a few. The focus of this paper is on autonomous linear systems, also referred to as vector autoregressive (VAR) models of order $1$. Here, the states $x_t \in \matR^d$ evolve with time as 
\begin{equation} \label{eq:single_lin_sys_intro}
    x_{t+1} = A^* x_{t} + \eta_{t+1}; \quad t=0,\dots,T; \ x_{0} = 0,
\end{equation}
with $A^*$ denoting the unknown system matrix to be estimated, and \rev{$(\eta_t)_{t \geq 1}$} denotes the (unobserved) process noise -- typically assumed to be centered and independent over $t$. While the problem has been studied extensively in the literature, classical work on linear system identification typically either consisted of asymptotic guarantees \rev{on the estimation error} (e.g., \citep{LaiWei82, LaiWei83, LaiWei1986}), or \rev{consisted of quantities which are} exponential in the degree of the linear system \citep{Campi2002, VidyaSagar2008}. 

%
\subsection{Finite-time identification of LDSs} \label{subsec:intro_finite_time_ident} 
A recent line of work has focused on finite-time identification of linear \rev{systems, which} amounts to deriving non-asymptotic error bounds for estimating $A^*$ (e.g., \citep{FaraUnstable18,Simchowitz18a, Sarkar19, Jedra20}). These results provide error bounds for estimating $A^*$ via the ordinary least-squares (OLS) estimator (yielding $\est{A}$), with the error quantified using the spectral norm $\norm{\est{A} - A^*}_2$. The error is shown to decay to zero at the (near-) parametric rate\footnote{Throughout, $\tilde{O}(\cdot)$ hides logarithmic factors.} $\tilde{O}(1/\sqrt{T})$ provided $T$ is suitably large enough, and under different assumptions on the spectral radius\footnote{See Section \ref{subsec:notation} for a description of notation.}  \rev{$\rho(A^*)$}. A typical assumption in this regard is that of \emph{stability}, which amounts to saying that either $\rho(A^*) < 1$ (strict stability), or even $\rho(A^*) \leq 1$ (marginal stability). The state-of-art result in this regard is arguably that of \cite{Sarkar19} -- when $A^*$ is marginally stable, it is shown that with probability at least $1-\delta$, 
\begin{align} \label{eq:sarkar_rakhlin_LDS_bd}
 \norm{\est{A} - A^*}_2 \lesssim \sqrt{\frac{d \log\left(\frac{1+\Tr(\Gamma_T(A^*))}{\delta}\right)}{T}} \quad \text{ if } \quad T \gtrsim d \log\left(\frac{1+\Tr(\Gamma_T(A^*))}{\delta} \right). 
\end{align}
Here $\Gamma_t(A)$ is the controllability Grammian of a $d \times d$ matrix $A$, defined as 
\begin{equation} \label{eq:control_gram}
    \Gamma_t(A) := \sum_{k=0}^{t} A^k (A^k)^\top
\end{equation}
for any integer $t \geq 0$. When $\rho(A^*) \leq 1$, then one can show that $\Tr(\Gamma_T(A^*))$ is at most polynomial in $T$, thus implying a  $\tilde{O}(1/\sqrt{T})$ rate for the error $\norm{\est{A} - A^*}_2$. Hence, to drive $\norm{\est{A} - A^*}_2$ below a specified threshold, we require $T$ to be polynomially large w.r.t $d$. We also remark that the OLS is a consistent estimator for only \emph{regular} matrices, i.e., $A^*$ for which the geometric multiplicity of eigenvalues outside the unit circle, is equal to one (e.g., \citep{Sarkar19}). The results in \citep{Sarkar19} provide error rates for estimating regular matrices with an otherwise arbitrary distribution of eigenvalues.

\subsection{Learning multiple LDSs under smoothness constraints} Our focus in this paper is on the setting where we are given samples of the trajectories of $m$ LDSs, governed by system matrices $A^*_1,\dots,A^*_m$, with the state of the $l$th system evolving as
\begin{equation} \label{eq:mult_linsys}
    x_{l,t+1} = A^*_{l} x_{l,t} + \eta_{l,t+1}; \quad t=0,\dots,T, \ x_{l,0} = 0.
\end{equation}
While we assume for convenience that the trajectories evolve independently of each other (see Section \ref{sec:prob_setup} for details), we place a smoothness assumption on $A^*_l$'s as dictated by a known underlying (undirected) network $G = ([m], \calE)$. More formally, we assume for $S_m \geq 0$ that
\begin{equation} \label{eq:quad_var_smooth}
    \sum_{\set{l,l'} \in \calE} \norm{A^*_l - A^*_{l'}}_F^2 \leq S_m
\end{equation}
which means that the quadratic variation of $A^*_1,\dots,A^*_m$ w.r.t $G$ is bounded. For bounded $A^*_l$'s one can obviously find \rev{$S_m = O(\abs{\calE})$} such that \eqref{eq:quad_var_smooth} is true, however the interesting smoothness regime is when \rev{$S_m = o(\abs{\calE})$}. Indeed, in this regime, the node-dynamics of the graph $G$ share a non-negligible amount of information between themselves, and thus one might hope to leverage this fact in order to jointly estimate $A^*_l$'s meaningfully. In this regard, the following ``sanity checks'' are useful.
\begin{itemize}
    \item The above setup is particularly meaningful when $T$ is small, thus preventing us from reliably estimating $A^*_l$'s on an individual basis, i.e., by only using the data $(x_{l,t})_{t=1}^{T+1}$. 

    \item Even if $T$ were large enough to be able to reliably estimate $A^*_l$'s individually (e.g., via OLS), one could still hope to obtain an even smaller estimation error via a suitable joint estimation procedure.
\end{itemize}
%
%
\paragraph{Contributions.} We propose two estimators for this problem, and derive non-asymptotic error bounds on the mean-squared error (MSE) $\frac{1}{m} \sum_{l=1}^m \norm{\est{A}_l - A^*_{l}}_F^2$, where $\est{A}_l$ is the estimate of $A^*_l$. \rev{For both the estimators, we establish conditions under which the MSE goes to $0$ as $m$ increases.}
\begin{enumerate}
    \item We first propose a smoothness penalized least-squares estimator that incorporates the quadratic variation in \eqref{eq:quad_var_smooth} as a penalty term, see Sections  \ref{subsec:lapl_smooth_algo} and \ref{subsec:lapl_smoothing_results}. For this estimator, our main result (see Theorem \ref{thm:lapl_smooth_strict_stab}) bounds the MSE with high probability for any connected $G$, provided $T$ is sufficiently large. In particular, the analysis requires $T$ to be at least as large as the \rev{term in \eqref{eq:sarkar_rakhlin_LDS_bd}} (maximized over $l \in [m]$), but with an additional additive $\log m$ factor as well. While this requirement on $T$ is an artefact of the analysis, we note that even so, our bounds on MSE are much smaller w.r.t $m$ than what we would obtain by estimating $A^*_l$'s individually (see Remark \ref{rem:laplsmooth_comp_sing_LDS_bd}). The above result is instantiated for different graph topologies with precise rates for the MSE when each $A^*_l$ is marginally stable (see Corollaries \ref{corr:lapl_smooth_strict_stab} and \ref{corr:lapl_smooth_strict_stab_compl_star}). 

    \item The second estimator we propose amounts to solving a least squares problem, constrained to the subspace spanned by the smallest few eigenvectors of the Laplacian of $G$ (see Section \ref{subsec:subs_contr_algo} and \ref{subsec:subspace_const_est_results}). Our main result for this estimator (see Theorem \ref{thm:thm_subs_constr_strict_stab}) bounds the MSE with high probability, and is applicable for the regime of small $T$ \emph{even} when $T = 2$.  On the other hand, the analysis applies for graphs $G$ for which the eigenvectors of the Laplacian are sufficiently ``delocalized'', i.e., the mass is approximately evenly spread out amongst the entries (Definition \ref{def:deloc_eigvecs}). Moreover, the result is only meaningful in the smoothness regime $S_m = o(m^{-1/3})$. The above result is instantiated for the path graph with precise rates for the MSE when each $A^*_l$ is marginally stable -- see Corollary \ref{corr:subs_constr_strict_stab_path} and the ensuing remarks for more details.
\end{enumerate}
In our simulations (see Section \ref{sec:sims}), we find that both the above estimators work well for small $T$. The proof technique for the above results leverages ideas from the work of \cite{Sarkar19} for learning marginally stable LDSs, along with traditional ideas from nonparametric regression. While Theorems \ref{thm:lapl_smooth_strict_stab} and \ref{thm:thm_subs_constr_strict_stab} are specifically applicable for the setup where each $A^*_l$'s is marginally stable, this is only for containing the length of the manuscript. The analysis can be extended to the more general setting involving regular matrices $A^*_l$, using ideas from \citep{Sarkar19}.

\subsection{Related work}
%
The problem of jointly learning multiple LDSs has received significant attention recently, especially due to many applications arising in modeling, e.g., the dynamics of brain-networks \citep{model_brain20, Gu2014ControllabilityOS}, the dynamics of flights at different altitudes \citep{Bosworth92}, gene expressions in genomics \citep{basu15a} etc. 

\cite{modi_jointlearn24} considered the setting where the system matrices $A^*_1,\dots,A^*_m$ share common basis matrices. More formally, for some $k$ (unknown) basis matrices $W_1,\dots,W_k \in \matR^{d \times d}$ it is assumed for each $l \in [m]$ that
%
        $A_{l} = \sum_{i=1}^k \beta_{l,i} W_i$
%
for some (unknown) scalars $\beta_{l,1}, \dots, \beta_{l,k}$. They propose an estimator for estimating the system matrices and derive bounds on the MSE. In particular, it is shown that reliable estimation can be performed provided $T > k$, which is meaningful when $k \ll d^2$.  

\cite{pmlr-v211-wang23d} consider a federated approach to learning multiple LDSs, where the $l$th system (or client) is modeled as 
\begin{equation*} 
    x_{l,t+1} = A^*_{l} x_{l,t} + B^*_{l} u_{l,t} +  \eta_{l,t+1}; \quad t=0,\dots,T.
\end{equation*}
Here $u_{l,t}$ represents external input to the system. It is assumed that the clients cannot communicate with each other, but only with a single server. Moreover, it is assumed that $N_l$ independent trajectories are observed for the $l$th client. The goal is to find a common estimate $\est{A}, \est{B}$ to all the system matrices $A^*_l, B^*_l$. To achieve this, it is assumed that 
\begin{equation} \label{eq:intro_fedsys_assump}
    \max_{l,l' \in [m]} \norm{A^*_l - A^*_{l'}}_2 \leq \epsilon, \quad \max_{l,l' \in [m]} \norm{B^*_l - B^*_{l'}}_2 \leq \epsilon.
\end{equation}
Then, $\est{A}, \est{B}$ are obtained via the OLS method and error bounds are derived on $\norm{\est{A}-A^*_l}_2, \norm{\est{B}-B^*_l}_2$, uniformly over $l$. Their main result states that if $N_l$ is suitably large (w.r.t $T$ and system dimensions), and $\epsilon$ is small enough, then the estimation error (for each $l$) is smaller than what would be obtained by individual estimation at each client $l$. Note that \eqref{eq:intro_fedsys_assump} implies \eqref{eq:quad_var_smooth} with $S_m = O(m^2 \epsilon^2)$ and is applicable when $G$ is the complete graph. But \eqref{eq:quad_var_smooth} is clearly a much weaker condition as it allows for $\norm{A^*_l - A^*_{l'}}_F = \Omega(1)$ for some edges $\set{l,l'}$.

\cite{Chen2023MultiTaskSI} considered the setting of jointly learning systems of the form \eqref{eq:intro_fedsys_assump} by placing different structural assumptions on $A^*_l$'s. One of these assumptions stipulates $\norm{A^*_l - A^*_{l'}}_F^2 \leq \epsilon$ for all $l,l' \in [m]$, which is again stronger than \eqref{eq:quad_var_smooth}. The estimation is then performed via a penalized least-squares approach, which is essentially our estimator in Section \ref{subsec:lapl_smooth_algo} when $G$ is the complete graph. No theoretical results are provided for this estimator.

\cite{Xin2022IdentifyingTD} considered the setting of learning a LDS where we are given given the trajectory of not only the true system, but also that of a ``similar'' system. A weighted least-squares approach is proposed for estimating the true system parameters and it is shown that the error can be reduced by leveraging data from the two systems.

We remark that when $S_m = 0$ in \eqref{eq:quad_var_smooth} then we have $A^*_l = A^*$ for each $l$, and our problem reduces to learning a single LDS from $m$ independent trajectories. There exist many results in this regard for estimating the system matrices,  with error bounds w.r.t the spectral norm (see e.g., \citep{Xin22MultTraj, zheng2020non}).  

Finally, we remark that our motivation for the smoothness constraint \eqref{eq:quad_var_smooth} comes from the literature on denoising smooth signals on an undirected graph $G = ([m], \calE)$. Here, we are given noisy measurements $y \in \matR^m$ of a signal $x^* \in \matR^m$, where 
$$y_l = x^*_l + \epsilon_l; \quad l=1,\dots,m$$
with $\epsilon_l$ denoting independent and centered  random noise samples. 
The signal $x^*$ is assumed to be smooth w.r.t $G$, i.e., $(x^*)^\top L x^* \leq S_m$ for some ``small'' $S_m$, with $L$ denoting the Laplacian of $G$. A common method for estimating $x^*$ is the so-called Laplacian smoothing estimator \citep{sadhanala2016total}, where for a regularizer $\lambda \geq 0$,
\begin{equation} \label{eq:graph_den_lapl_smooth}
    \est{x} = \argmin{x \in \matR^m} \set{ \norm{x-y}_2^2 + \lambda x^\top L x} = (I_m + \lambda L)^{-1} y.
\end{equation}
Error rates for $\expec[\norm{\est{x} - x^*}_2^2]$ were obtained in \citep{sadhanala2016total} (for the optimal choice of $\lambda$) when $G$ is a multidimensional grid. To put things in perspective, note that the data available at each node in our setting is a vector-valued \emph{time-series} generated as in \eqref{eq:mult_linsys}, while the above setting typically has a single scalar-valued observation at each node. Although our estimator in Section \ref{subsec:lapl_smooth_algo} is motivated from \eqref{eq:graph_den_lapl_smooth}, its analysis is considerably more challenging than that of \eqref{eq:graph_den_lapl_smooth}.

%% file: prob_setup.tex
\section{Problem setup} \label{sec:prob_setup}
\subsection{Notation} \label{subsec:notation}
For any vector $x \in \matR^n$, $\norm{x}_p$ denotes the usual $\ell_p$ norm of $x$. For $X \in \matR^{n \times m}$, we denote \rev{the spectral and Frobenius norms of $X$ by $\norm{X}_2$ and $\norm{X}_F$ respectively}, while $\dotprod{X}{Y} = \Tr(X^\top Y)$ denotes the inner product between $X$ and $Y$ (with $\Tr(\cdot)$ denoting the trace operator). For $n \times n$ matrices X with eigenvalues $\lambda_i \in \mathbb{C}$, we denote \rev{the spectral radius of $X$ by $\rho(X) := \max_{i=1,\dots,n} \abs{\lambda_i}$}. Also, $\kappa(X) (\geq 1)$ denotes the condition number of $X$ (ratio of largest and smallest singular values).
For a positive definite matrix $B$, and a vector $x$, we will denote $$\norm{x}_{B} := \sqrt{x^\top B x} = \norm{B^{1/2} x}_2.$$

We denote the identity matrix by $I_n$. Also, $\vect(X) \in \matR^{nm}$ is formed by stacking the columns of $X$. The symbol $\otimes$ denotes the Kronecker product between matrices. 

For $a,b > 0$, we say $a \lesssim b$ if there exists a constant $C > 0$ such that $a \leq C b$. If $a \lesssim b$ and $a \gtrsim b$, then we write $a \asymp b$.  The values of symbols used for denoting constants (e.g., $c, C, c_1$ etc.) may change from line to line. It will sometimes be convenient to use asymptotic notation. For non-negative functions $f(x), g(x)$, we write   
\begin{itemize}
\item $f(x) = O(g(x))$ if there exists $c > 0$ and $x_0$ such that $f(x) \leq c g(x)$ for $x \geq x_0$;

\item $f(x) = \Omega(g(x))$ if $g(x) = O(f(x))$;

\item $f(x) = \Theta(g(x))$ if there exists $c_1, c_2 > 0$ and $x_0$ such that $c_1 g(x) \leq f(x) \leq c_2 g(x)$ for $x \geq x_0$.
\end{itemize}
Furthermore, we write $f(x) = o(g(x))$ if $\lim_{x \rightarrow \infty} \frac{f(x)}{g(x)} = 0$, and $f(x) = \omega(g(x))$ if $\lim_{x \rightarrow \infty} \frac{f(x)}{g(x)} = \infty$.

Recall that the subgaussian norm of a random variable $X$ is given by 
$$\norm{X}_{\psi_2}:= \sup_{p \geq 1} p^{-1/2} (\expec \abs{X}^p)^{1/p},$$ see for e.g. \citep{vershynin_2012} for other equivalent definitions. When $X$ is centered, we will say that it is $R$-subgaussian if 
\begin{equation*}
    \expec[\exp(\beta X)] \leq \exp\Big(\frac{\beta^2 R^2}{2}\Big), \quad \forall \ \beta \in \matR.
\end{equation*}
This means that $\norm{X}_{\psi_2} \leq c R$ for some universal constant $c > 0$ \citep[Definition 5.7]{vershynin_2012}. 

Finally, we denote \rev{the indicator function for an event $\calB$ by $\indic_{\calB}$}.

\subsection{Setup} \label{subsec:setup}
Let $A^*_1,\dots,A^*_m \in \matR^{d \times d}$ be unknown system matrices. For each $l \in [m]$, consider the linear dynamical system (LDS) in \eqref{eq:mult_linsys}
%
%
where $x_{l,t}, \eta_{l,t} \in \matR^{d}$. Given the samples $(x_{l,1},\dots,x_{l,T+1})$ for each $l \in [m]$, our goal is to estimate $A^*_1,\dots,A^*_m$. 
\begin{assumption} \label{eq:eta_assump}
    We will assume that $\eta_{l,t}$ are i.i.d centered random vectors with independent coordinates (for each $l,t$). Furthermore, each coordinate of $\eta_{l,t}$ is assumed to have unit variance.
\end{assumption}
Further assumptions on the distribution of $\eta_{l,t}$ will be made later when we present our results.
The above setup corresponds to multiple linear dynamical systems, $m$ in total, with the dynamics of the $l$th system governed by \eqref{eq:mult_linsys}. Without further assumptions that in some sense relate the systems to each other, we clearly cannot do better than simply estimating each $A^*_l$ individually from the data $x_{l,1}, \dots,x_{l,T+1}$. To this end, denoting $G = ([m], \calE)$ to be an underlying connected graph, with the state of the $l$th node evolving as per \eqref{eq:mult_linsys}, we will assume that $A^*_1,\dots,A^*_m$ are smooth w.r.t $G$ in the following sense.
\begin{assumption}[Smoothness] \label{assum:smoothness}
For a given $G = ([m], \calE)$, the matrices $A^*_1,\dots,A^*_m$ satisfy \eqref{eq:quad_var_smooth} 
%
%
for some $S_m \geq 0$.
\end{assumption}
Assumption \ref{assum:smoothness} is a global smoothness assumption on the matrices, and states roughly that $A^*_l$ is close to $A^{*}_{l'}$ (on average) over the edges $\set{l,l'} \in \calE$. Note that $S_m = O(m)$ is always true, however the interesting regime is when $S_m = o(m)$. The LHS of \eqref{eq:quad_var_smooth} corresponds to the quadratic variation of $(A^*_1,\dots,A^*_m)$ w.r.t the Laplacian matrix $L \in \matR^{m \times m}$ of $G$. Indeed, denoting $a^*_l = \vect(A^*_l) \in \matR^{d^2\times 1}$, and $a^* \in \matR^{md^2 \times 1}$ (formed by column-wise stacking of $a^*_l$'s), \eqref{eq:quad_var_smooth} corresponds to 
\begin{equation} \label{eq:vect_form_smooth}
   \sum_{\set{l,l'} \in \calE} \norm{A^*_l - A^*_{l'}}_F^2 = \rev{{(a^*)}^{\top} } (L \otimes I_{d^2}) a^* \leq S_m.
\end{equation}
In what follows, the eigenvalues of $L$ are denoted by $\lambda_1 \geq \cdots \geq \lambda_{m-1} > \lambda_m = 0$, and the corresponding (orthonormal) eigenvectors by $v_1,\dots,v_m$.

%
\begin{example} \label{ex:path_graph_holder}
For any $1 \leq i,j \leq d$, let each $A^*_{l,i,j}$ (the $(i,j)th$ entry of $A^*_l$) be formed by uniform sampling of a $(M, \beta)$-H\"older continuous function $f^*_{i,j}: [0,1] \rightarrow \matR$, for $M > 0$ and $\beta \in (0,1]$, i.e.,
\begin{equation*}
    A^*_{l,i,j} = f^*_{i,j}(l/m); \quad l=1,\dots,m
\end{equation*}
where $\abs{f^*_{i,j}(x) - f^*_{i,j}(x)} \leq  M \abs{x - y}^{\beta}$ for all $x,y \in [0,1]$. Then with $G$ denoting the path graph on $[m]$ where we recall $\calE = \set{\set{i,i+1}: i=1,\dots,m-1}$, clearly \eqref{eq:quad_var_smooth} is  satisfied with $S_m = M^2 d^2 m^{1-2\beta}$. Hence, assuming $M,d$ are fixed, we have $S_m = O(m^{1-2\beta})$ as $m$ increases. In the Lipschitz case where $\beta = 1$, we arrive at $S_m = O(m^{-1})$.
\end{example}

Our aim is to jointly estimate $(A^*_{l})_{l=1}^m$ under Assumption \ref{assum:smoothness}, and derive conditions under which the estimates $(\est{A}_l)_{l=1}^m$ satisfy the \emph{weak-consistency} condition
\begin{equation} \label{eq:weak_consis}
\text{(MSE)} \quad \frac{1}{m} \sum_{l=1}^m \norm{\est{A}_l - A^*_l}_F^2 \xrightarrow{m \rightarrow \infty} 0.
\end{equation}
This is especially meaningful when $T$ is small, in which case estimating $A^*_l$ solely from its trajectory $(x^*_{l,t})_{t=0}^T$ is not possible. However it is also relevant when $T$ is suitably large, possibly growing mildly with $m$, provided the MSE under joint estimation is smaller than what we would obtain by individual estimation of the system matrices. We will study two estimators for this problem that we proceed to describe next. 
%
%
\subsection{Laplacian smoothing} \label{subsec:lapl_smooth_algo}
The first \rev{estimator} is the penalized least-squares (LS) method 
\begin{equation} \label{eq:pen_ls_algo}
    (\est{A}_1,\dots,\est{A}_m) = \argmin{A_1,\dots,A_m \in \matR^{d \times d}} \ \left(\rev{\sum_{l=1}^m \sum_{t=1}^T \norm{x_{l,{t+1}} - A_l x_{l,t}}_2^2} + \lambda \sum_{\set{l,l'} \in \calE} \norm{A_{l} - A_{l'}}_F^2 \right)
\end{equation}
where $\lambda \geq 0$ is a regularization parameter. Denoting $a_l = \vect(A_l)$, we form the vector $a \in \matR^{md^2}$ by column-wise stacking of $a_1,\dots,a_m$. Furthermore, denoting the quantities 
\begin{align*}
    \Xtil_l = [x_{l,2} \cdots x_{l,T+1}], \ X_l = [x_{l,1} \cdots x_{l,T}], \ \xtil_l = \vect(\Xtil_l), 
\end{align*}
form $\xtil \in \matR^{dmT}$ by stacking $\xtil_l$'s, and $Q = \blkdiag(X_l^\top \otimes I_d)_{l=1}^m$. Then, \eqref{eq:pen_ls_algo} can be re-written as 
\begin{equation} \label{eq:vec_pen_ls_algo}
 (\est{a}_1,\dots,\est{a}_m) = \argmin{a_1,\dots,a_m} \ \left(\norm{\xtil - Q a}_2^2 + \lambda a^{\top} (L \otimes I_{d^2}) a \right).
\end{equation}
The above estimator is analogous to the Laplacian-smoothing estimator for denoising smooth signals on a graph (e.g., \citep{sadhanala2016total}).

\subsection{Subspace-constrained LS} \label{subsec:subs_contr_algo}
To motivate the second estimator, consider for any integer $1 \leq \tau \leq m-1$  
\begin{equation*}
    L \otimes I_{d^2} = \left(\sum_{l=m-\tau+1}^m \lambda_l v_l v_l^{\top} \right) \otimes I_{d^2} + \left( \sum_{l=1}^{m - \tau}\lambda_l v_l v_l^{\top} \right) \otimes I_{d^2}.
\end{equation*}
\rev{Denoting $\Pproj = \sum_{l = m - \tau + 1}^{m} v_l v_l^{\top}$, which is the} projection matrix for the low-frequency part of $L$, we obtain  another projection matrix $\Ptilproj = \Pproj \otimes I_{d^2}$. The condition \eqref{eq:vect_form_smooth} implies that $a^{*} \in \matR^{md^2}$ -- formed by column-wise stacking of $\vect(A^*_1),\dots,\vect(A^*_m)$ -- lies close to the low frequency part of $L \otimes I_{d^2}$. Hence, $a^{*} \approx \Ptilproj a^*$ for a suitable $\tau$ depending on $S_m$. For example, if $S_m = 0$, then $\tau = 1$ and clearly $a^{*} = \Ptilproj a^*$.  

Motivated by the above observation, we define our estimator $\est{a}$ as the solution of 
\begin{equation} \label{eq:proj_est_1}
  \argmin{a \in \text{span}(\Ptilproj)}  \norm{\xtil - Q a}_2^2 \equiv  \argmin{a \in \matR^{md^2 }} \  \norm{\xtil - Q \Ptilproj a}_2^2.
\end{equation}
Since the above objective function has infinitely many minimizers, we define $\est{a}$ as the minimum norm solution of \eqref{eq:proj_est_1} leading to
\begin{equation} \label{eq:proj_est_closed_form}
    \est{a} := [\Ptilproj Q^\top Q \Ptilproj]^{\dagger} \Ptilproj Q^{\top} \xtil 
\end{equation}
where $\dagger$ denotes the pseudo-inverse.
\begin{remark} \label{rem:alt_forms_multlinsys}
    Notice that the model \eqref{eq:mult_linsys} can be written as a single autonomous LDS as
    \begin{equation} \label{eq:alt_forms_multlinsys}
        \begin{bmatrix}
   x_{1,t+1} \\ 
   x_{2,t+1} \\ 
  \vdots \\
  x_{m,t+1}
 \end{bmatrix} =
\underbrace{\begin{bmatrix}
   A_1^* & 0 & \hdots & 0\\ 
  0 & A_2^* & \hdots & 0\\ 
  \vdots  &  & \ddots & \vdots\\
  0 & \hdots & \hdots & A_m^*
 \end{bmatrix}}_{ =: A^*}
\begin{bmatrix}
   x_{1,t} \\ 
   x_{2,t} \\ 
  \vdots \\
  x_{m,t}
\end{bmatrix}
+ 
\begin{bmatrix}
   \eta_{1,t+1} \\ 
   \eta_{2,t+1} \\ 
  \vdots \\
  \eta_{m,t+1}
 \end{bmatrix}.
\end{equation}
where $A^* \in \matR^{md \times md}$ is the unknown matrix to be estimated. In fact, $A^*$ is a structured matrix: its off-block-diagonal entries are zero (which mounts to linear equality constraints) and moreover, it also lies within the ellipsoid specified by the constraint \eqref{eq:quad_var_smooth}. Thus $A^*$ belongs to a closed convex set $\calK \subset \matR^{md \times md}$, and so, one could also estimate it via a $\calK$-constrained least squares estimator (which is a convex program). This could then be analyzed, e.g., following the approach of \cite{tyagi2024learning} for structured LDSs (but would likely lead to weaker results as the analysis therein pertains to general sets $\calK$). The formulation in \eqref{eq:vec_pen_ls_algo} has a closed-form solution which makes it easier to analyze.
\end{remark}

%% file: main_results.tex
\section{Main results} \label{sec:main_results}
Recall from \eqref{eq:control_gram} the definition of the controllability Grammian of a LDS with system matrix $A$.

\subsection{Laplacian smoothing} \label{subsec:lapl_smoothing_results}
Our main result for the penalized LS estimator is outlined below. The proofs of all results in this section are detailed in Section \ref{sec:laplacian_smooth_analysis}.
\begin{theorem} \label{thm:lapl_smooth_strict_stab}
Under Assumption \ref{eq:eta_assump}, suppose  $\eta_{l,t}$ has $R$-subgaussian entries for each $l,t$. Furthermore, let $G = ([m], \calE)$ be any connected graph. Then there exist constants $c, C, C_1 > 0$ such that the following is true. For $\delta \in (0,e^{-c})$, if 
\begin{equation} \label{eq:T_cond_Lapsmooth_strict_stab}
    T \geq C\max \set{R^4 \left(d + \log\left(\frac{m}{\delta} \right) \right), R^2 d\max_{l \in [m]} \set{\log(\Tr(\Gamma_T(A_l^*) - I) + 1) + \log\left(\frac{m}{\delta} \right)}}
\end{equation}
we have w.p (with probability) at least $1-3\delta$ that the solution of \eqref{eq:pen_ls_algo} is unique, and satisfies
\begin{align} \label{eq:lapl_smooth_mainbd}
    \frac{1}{m} \sum_{l=1}^m \norm{\est{A}_l - A_l^*}_F^2 &\leq \frac{C_1}{m}\Bigg(\frac{\lambda^2}{T^2} \norm{(L \otimes I_{d^2}) a^*}_2^2 + \frac{R^2}{T} \left[ \log\left(\frac{1}{\delta}\right) + d^2 \log\left( 1 + \frac{\gamma_1 (T,m,\delta)}{T}\right) \right] \nonumber \\
    &+ \frac{R^2 d^2 \gamma_1 (T,m,\delta)}{T} \left[ \sum_{l=1}^{m-1} \frac{1}{4\lambda \lambda_l + T} \right] \Bigg),
\end{align}
where 
\begin{equation} \label{eq:gamma_def_laplsmooth_strict_stab}
    \gamma_1(T,m,\delta) := \left(1 + R^2 \log\left(\frac{m}{\delta} \right) \right) \max_{l \in [m]} \set{\sum_{t=0}^{T-1} \Tr (\Gamma_t(A_l^*))}.
\end{equation}
\end{theorem}
The proof is motivated by ideas in the proof of \cite[Theorem 1]{Sarkar19} for learning marginally stable linear dynamical systems. \rev{The following points are useful to note.}
\begin{itemize}
\item The first error term in \eqref{eq:lapl_smooth_mainbd} corresponds to the bias (\rev{which increases as $\lambda$ increases}), while the remaining two terms therein constitute the variance (\rev{which decrease as $\lambda$ increases}). \rev{The bias-variance trade-off will be achieved by choosing $\lambda$ to minimize the sum of these terms.} 

\item The condition on $T$ in \eqref{eq:T_cond_Lapsmooth_strict_stab} is an artefact of the analysis, and arises from Lemma \ref{lem:laplace_smooth_qq_lowbd} in order to ensure $Q^\top Q + \lambda (L \otimes I_{d^2}) \succ 0$. In general, obtaining lower bounds on the eigenvalues of  
$Q^\top Q + \lambda (L \otimes I_{d^2})$ is challenging as the eigenspaces of $L \otimes I_{d^2}$ and $Q^\top Q$ are not aligned. Our strategy is the admittedly naive approach that instead seeks to ensure $Q^\top Q \succ 0$ using results from \cite[Section 9]{Sarkar19}.
\end{itemize}
It will be instructive now to explicitly consider each $A^*_l$ to be marginally stable, i.e., 
\begin{equation} \label{eq:stable_mat_specrad}
    \rho(A^*_l) \leq 1 + \frac{c}{T}; \quad l=1,\dots,m,
\end{equation}
for some constant $c > 0$, and to also consider particular choices of graphs $G$. This would enable us to bound the terms $\Tr(\Gamma_T(A_l^*))$ (and hence $\gamma_1(T,m,\delta)$), and also illustrate choosing $\lambda$ appropriately in order to balance the bias-variance terms such that $\frac{1}{m} \sum_{l=1}^m \norm{\est{A}_l - A_l^*}_F^2 = o(1)$ as $m \rightarrow \infty$.
\begin{remark} \label{rem:stab_assump_mats}
    Note that Theorem \ref{thm:lapl_smooth_strict_stab} does not require a stability condition on $A^*_l$'s. Also, \eqref{eq:stable_mat_specrad} is equivalent to saying that the matrix $A^*$ in Remark \ref{rem:alt_forms_multlinsys} is marginally stable. If $A^*_l$'s are allowed to be explosive, i.e., have eigenvalues outside the unit circle, then a different analysis would be required, as detailed in \citep{Sarkar19}. In fact, it would be interesting to verify whether Laplacian smoothing is weakly consistent even if some of the $A^*_l$'s are irregular.
\end{remark}
Before proceeding, let $A^*_l = P_l^{-1} \Lambda_l P_l$ be the Jordan Canonical Form of $A^*_l$, and recall that $\kappa(\cdot)$ denotes the condition number of a matrix. Then as a consequence of Proposition \ref{prop:stable_mat_trace_bound}, we have for any $T \geq 2$ that 
\begin{equation} \label{eq:trace_gamma_bd_stabmat}
    \Tr (\Gamma_T(A_l^*)) \lesssim d T^{\alpha_l} \kappa^2(P_l)
\end{equation}
for some $\alpha_l \in [2d]$ which is proportional to the size of the largest Jordan block of $A^*_l$. For instance, if $A^*_l$ is symmetric, then we have $\alpha_l = 1 = \kappa(P_l)$. It will be useful to define the quantities
\begin{equation} \label{eq:alpha_kappa_defs_stabmat}
    \alpha := \max_{l=1,\dots,m} \alpha_l, \quad \kappa_{\max} := \max_{l=1,\dots,m} \kappa(P_l)
\end{equation}
We then obtain the following corollary of Theorem \ref{thm:lapl_smooth_strict_stab} when $G$ is the path graph and each $A^*_l$ is marginally stable.
%
%
\begin{corollary}[Path graph] \label{corr:lapl_smooth_strict_stab}
Under Assumptions \ref{eq:eta_assump} and \ref{assum:smoothness}, suppose  $\eta_{l,t}$ has $R$-subgaussian entries for each $l,t$. Furthermore, let $G = ([m], \calE)$ be the path graph and choose $\lambda = (Rd)^{4/5} (\frac{m}{S_m})^{2/5} T^{1/5}$. Then there exist constants $c, C, C_1, C_2, C_3 > 0$ such that the following is true. 
\begin{enumerate}
\item For $\delta \in (0,e^{-c})$, if
$T$ satisfies \eqref{eq:T_cond_Lapsmooth_strict_stab}, then $w.p$ at least $1-3\delta$ the solution of \eqref{eq:pen_ls_algo} is unique, and satisfies
\begin{align} \label{eq:lapl_smooth_path_graph_mainbd_1}
    \frac{1}{m} \sum_{l=1}^m \norm{\est{A}_l - A_l^*}_F^2 
    &\leq 
    C_1\Bigg(\Big(\frac{S_m}{m} \Big)^{1/5} \frac{(Rd)^{8/5} \gamma_1 (T,m,\delta)}{T^{8/5}} \nonumber \\ 
    &+ \frac{R^2}{m T} \left[ \log\left(\frac{1}{\delta}\right) + d^2 \log\left( 1 + \frac{\gamma_1 (T,m,\delta)}{T}\right) \right] \Bigg).
\end{align}
Here, $\gamma_1(T,m,\delta) $ is as in \eqref{eq:gamma_def_laplsmooth_strict_stab}.

\item Suppose each $A^*_l$ is marginally stable, i.e. satisfies \eqref{eq:stable_mat_specrad}. Then \eqref{eq:T_cond_Lapsmooth_strict_stab} is ensured provided 
\begin{equation} \label{eq:T_cond_laplsmooth_corr_path_stab_mat}
T \geq C_2 R^4 d \log\Big(\frac{d T^{\alpha} \kappa^2_{\max} m}{\delta}\Big),
\end{equation}
with $\alpha, \kappa_{\max}$ as in \eqref{eq:alpha_kappa_defs_stabmat}. Furthermore,  \eqref{eq:lapl_smooth_path_graph_mainbd_1} simplifies to
\begin{align} \label{eq:lapl_smooth_path_graph_mainbd_2}
  \frac{1}{m} \sum_{l=1}^m \norm{\est{A}_l - A_l^*}_F^2 
    &\leq 
    C_3 \Bigg(\Big(\frac{S_m}{m} \Big)^{1/5} (R)^{18/5} d^{\frac{13}{5}} \log\Big(\frac{m}{\delta} \Big) T^{\alpha-\frac{3}{8}} \kappa^2_{\max} \nonumber \\ 
    &+ \frac{R^2}{m T} \Big[ \log\left(\frac{1}{\delta}\right) + d^2 \log\left( d R^2 \log\Big(\frac{m}{\delta}\Big) T^{\alpha} \kappa^2_{\max} \right) \Big] \Bigg).
\end{align}
\end{enumerate}
\end{corollary}
%
%
The following comments are in order.
\begin{enumerate}
\item \rev{The $T^{\alpha}$ term appearing in the numerator of the first term in \eqref{eq:lapl_smooth_path_graph_mainbd_2} is essentially due to the $\gamma_1(T,m,\delta)$ term appearing outside the logarithm in the third term of \eqref{eq:lapl_smooth_mainbd}. This is an artefact of the analysis, and it is unclear how to improve this for the case of marginally stable matrices. But when $A^*_{l}$'s satisfy a stricter stability condition where $\norm{A^*_l}_2 < 1$ for each $l$, then this issue disappears. See Remark \ref{rem:exp_dep_T} below for details.}

  \item The term $(\frac{S_m}{m})^{1/5}$ is the non-parametric rate while the $\tilde{O}(1/m)$ term in \eqref{eq:lapl_smooth_path_graph_mainbd_1} is the parametric rate. The exponent $1/5$ is not optimal -- we believe this should be $1/3$ based on the results in \cite{sadhanala2016total} for $1$-d grids. The reason for the sub-optimal exponent is that the bias term scales quadratically with $\lambda$ in \eqref{eq:lapl_smooth_mainbd}, which in turn is linked with the fact that the eigenspaces of $Q^\top Q$ and $L \otimes I_{d^2}$ are not aligned (as remarked earlier). This leads to a sub-optimal estimate of the bias error in Lemma \ref{lem:bd_laplsmooth_e_1}.
  
    \item Treating $R, \kappa_{\max}$ as absolute constants for simplicity, we see that if the trajectory-length of each node is $T = \Omega(d \log(\frac{dmT^{\alpha}}{\delta}))$ then the MSE is 
    $$O\Bigg(\Big(\frac{S_m}{m}\Big)^{1/5} d^{13/5} T^{\alpha} \log\Big(\frac{m}{\delta} \Big) + \frac{1}{mT} \Big[ d^2 \log\Big(d\log\Big(\frac{m}{\delta}\Big) T^{\alpha}\Big)\Big] \Bigg).$$
    Hence if \rev{$S_m = O(m^{1-\epsilon})$} for some constant $\epsilon \in (0,1)$, then we require $m = \Omega(\poly(d,T^{\alpha}))$ to drive the MSE below a given threshold. Recall that $\alpha \in [2d]$, and depends on the size of the largest Jordan blocks of $A^*_l$'s. If for each $l$, the largest Jordan block of $A^*_l$ is of constant size, then $\alpha \asymp 1$. But in case $\alpha \asymp d$, then $m$ will grow exponentially with $d$ in our bounds. 
\end{enumerate}
The following corollaries of Theorem \ref{thm:lapl_smooth_strict_stab} are obtained for the setting of a complete graph and a star graph (with each $A^*_l$ marginally stable).
%
%
\begin{corollary}[Complete and Star graph] \label{corr:lapl_smooth_strict_stab_compl_star}
Suppose each $A^*_l$ is marginally stable, \rev{i.e.}, satisfies \eqref{eq:stable_mat_specrad}, and recall $\alpha,\kappa_{\max}$ in \eqref{eq:alpha_kappa_defs_stabmat}. Under Assumptions \ref{eq:eta_assump} and \ref{assum:smoothness}, suppose  $\eta_{l,t}$ has $R$-subgaussian entries for each $l,t$. 
\begin{enumerate}
\item (Complete graph) Let $G = ([m], \calE)$ be the complete graph with $\calE = \set{\set{i,j}: i \neq j \in [m]}$ and choose $\lambda =  (\frac{T R^2 d^2}{m S_m})^{1/3}$. There exist constants $c, C > 0$ such that for $\delta \in (0,e^{-c})$, if
$T$ satisfies \eqref{eq:T_cond_laplsmooth_corr_path_stab_mat}, then $w.p$ at least $1-3\delta$ the solution of \eqref{eq:pen_ls_algo} is unique, and satisfies
%
    %
    %
%
\begin{align} \label{eq:lapl_smooth_compl_graph_mainbd_1}
    \frac{1}{m} \sum_{l=1}^m \norm{\est{A}_l - A_l^*}_F^2 
    &\leq 
    C \Bigg( R^{10/3} d^{7/3} T^{\alpha-\frac{1}{3}}  \kappa^2_{\max} \Big(\frac{S_m^{1/3}}{m^{2/3}} \Big) \log(\frac{m}{\delta}) \nonumber \\ 
    &+ \frac{R^2}{m T} \Big[ \log\left(\frac{1}{\delta}\right) + d^2 \log\left( d R^2 \log\Big(\frac{m}{\delta}\Big) T^{\alpha} \kappa^2_{\max} \right) \Big] \Bigg).
\end{align}

\item (Star graph) Let $G = ([m], \calE)$ be the star graph with $\calE = \set{\set{1,i}: 2 \leq i \leq m}$, and choose $\lambda =  (\frac{T R^2 d^2 m}{ \tilde{S}_m})^{1/3}$ where 
\begin{equation*}
    \tilde{S}_m := S_m + \rev{m^2 \Big\|A_1^* - \frac{1}{m-1}\sum_{l=2}^{m} A_l^* \Big\|_F^2}.
\end{equation*}
There exist constants $c, C > 0$ such that for $\delta \in (0,e^{-c})$, if
$T$ satisfies \eqref{eq:T_cond_laplsmooth_corr_path_stab_mat}, then $w.p$ at least $1-3\delta$ the solution of \eqref{eq:pen_ls_algo} is unique, and satisfies
%
    %
    %
%
%
\begin{align} \label{eq:lapl_smooth_star_graph_mainbd_1}
    \frac{1}{m} \sum_{l=1}^m \norm{\est{A}_l - A_l^*}_F^2 
    &\leq 
    C \Bigg( R^{10/3} d^{7/3} T^{\alpha-\frac{1}{3}} \kappa^2_{\max} \Big(\frac{\tilde{S}_m}{m}\Big)^{1/3} \log(\frac{m}{\delta})   \nonumber \\ 
    &+ \frac{R^2}{m T} \Big[ \log\left(\frac{1}{\delta}\right) + d^2 \log\left( d R^2 \log\Big(\frac{m}{\delta}\Big) T^{\alpha} \kappa^2_{\max} \right) \Big] \Bigg).
\end{align}
\end{enumerate}
\end{corollary}
\rev{Let us interpret the above corollary.}
\begin{itemize}
    \item \rev{The requirement $T = \Omega(d \log(\frac{dmT^{\alpha}}{\delta}))$ (for both  complete and star graphs) is the same as that for a path graph. Also, the error bounds for both the examples contain the term $T^{\alpha}$ in the numerator (as in the case of the path graph) which again is an artefact of the analysis. As mentioned earlier, this issue disappears if $\norm{A^*_l}_2 < 1$ for each $l$. See Remark \ref{rem:exp_dep_T} below.} 

    \item \rev{When $G$ is the complete graph, the interesting regime for the smoothness is $S_m = o(m^2)$. If $S_m = O(m^{2-\epsilon})$ for some 
    $\epsilon$ in  $(0,2)$, then we require $m = \Omega(\poly(d,T^{\alpha}))$ to drive the MSE below any given threshold. As in the case of a path graph, this bound will be exponential in $d$ if $\alpha \asymp d$.}

    \item \rev{When $G$ is the star graph, the term $\tilde{S}_m$ captures the smoothness of the problem. Not surprisingly, $\tilde{S}_m$ is small when a small fraction of the degree $1$ vertices have matrices which are far from $A^*_1$ (the matrix at the ``central'' node). If $\tilde{S}_m = O(m^{1-\epsilon})$ for $\epsilon$ in $(0,1)$, then we again require $m = \Omega(\poly(d,T^{\alpha}))$ (as in the other examples) to make the MSE small.} 
\end{itemize}
%
%
%
%
\begin{remark}[Dependence on $T$ under stricter stability conditions] \label{rem:exp_dep_T}
\rev{The error bounds in Corollaries \ref{corr:lapl_smooth_strict_stab} and \ref{corr:lapl_smooth_strict_stab_compl_star} (for the marginally stable case) involve the term $T^{\alpha}$ appearing in the numerator of the first term. As discussed earlier, this is an artefact of the analysis. However, in the relatively easier setting\footnote{\rev{which implies each $A^*_l$ is strictly stable as $\rho(A^*_l) \leq \norm{A^*_l}_2$.}} where each $A^*_l$ satisfies $\norm{A^*_l}_2 < 1$, this issue disappears. Indeed, by noting that }
%
\begin{align*}
   \rev{\Tr(\Gamma_T(A^*_l)) = \sum_{k=0}^T \Tr(({A^*_l})^k (({A^*_l})^k)^\top) 
    \leq \sum_{k=0}^T d \norm{({A^*_l})^k (({A^*_l})^k)^\top}_2 \leq d \sum_{k=0}^T\norm{A^*_l}_2^{2k} \leq \frac{d}{1-\norm{A^*_l}_2^2}}
\end{align*}
\rev{and comparing it with \eqref{eq:trace_gamma_bd_stabmat}, we can simply set 
\begin{equation*}
\alpha = 0 \quad \text{and} \quad \kappa^2_{\max} = (1-\max_l \norm{A^*_l}_2^2)^{-1} 
\end{equation*}
in Corollaries \ref{corr:lapl_smooth_strict_stab} and \ref{corr:lapl_smooth_strict_stab_compl_star} to obtain the error bounds in this setup. In particular, notice that the error bounds now go to zero as $T$ increases, which is of course what we would ideally like to have. Moreover, we now require $m = \Omega(\poly(d))$ to drive the MSE below a threshold.}
\end{remark}

\begin{remark} \label{rem:laplsmooth_comp_sing_LDS_bd}
As a sanity check, let us verify that we obtain better bounds than what we would have by naively estimating each $A^*_l$ individually. Indeed, assuming $R$ to be a constant, if \eqref{eq:stable_mat_specrad} holds, then applying the OLS for each node $l$, we obtain using \eqref{eq:sarkar_rakhlin_LDS_bd}, \eqref{eq:trace_gamma_bd_stabmat}, and a simple union bound that with probability at least $1-\delta$,   
\begin{align} \label{eq:sing_LDS_bd}
 \frac{1}{m} \sum_{l=1}^m \norm{\est{A} - A^*_l}_F^2 \lesssim  \frac{d^2}{mT} \sum_{l=1}^m \log\Big(\frac{m(1+\Tr(\Gamma_T(A^*_l)))}{\delta} \Big)  \lesssim \rev{\frac{d^2}{T} \log\Big(\frac{m d T^{\alpha} \kappa_{\max}^2}{\delta} \Big)}, 
\end{align}
provided $T \gtrsim d \max_{l \in [m]}\log(\frac{m[1+\Tr(\Gamma_T(A^*_l))]}{\delta})$. This condition on $T$ is essentially the same as in \eqref{eq:T_cond_Lapsmooth_strict_stab}, and is ensured for \rev{$T = \Omega(d \log(\frac{dmT^{\alpha}}{\delta}))$}. \rev{This latter condition on $T$ is ensured if $m \gtrsim d T^{\alpha}$ and $T = \Omega(d \log(m/\delta))$}. However in the regime  $T = \Theta(\log m)$, the bound in \eqref{eq:sing_LDS_bd} is $\Omega(1)$ w.r.t $m$.
\end{remark}
%
%
%

\subsection{Subspace-constrained LS} \label{subsec:subspace_const_est_results}
Our analyses will require that the eigenvectors of the Laplacian $L$ of $G$ are sufficiently ``delocalized'', i.e., the mass is sufficiently spread out amongst the entries of each eigenvector. 
\begin{definition} \label{def:deloc_eigvecs}
    We say $G$ is $\theta$-delocalized if there exists an orthonormal collection of eigenvectors $\set{v_l}_{l=1}^{m}$ of $L$ (with eigenvalues $\lambda_1 \geq \cdots \geq \lambda_{m-1} > \lambda_m = 0$) such that
\begin{equation} \label{eq:theta_deloc_cond}
\sum_{i=1}^\tau v_{m-i+1,l}^2 \leq \theta \frac{\tau}{m} \quad \forall \tau,l \in [m].
\end{equation}
\end{definition}
As we will see below, our results will be meaningful for graphs which satisfy \eqref{eq:theta_deloc_cond} with small value of $\theta$, e.g., a constant. This is the case, for instance,  for the path graph for which $\theta = 2$ (see proof of Corollary \ref{corr:subs_constr_strict_stab_path}). 
The following theorem is our main result for the subspace-constrained estimator \eqref{eq:proj_est_1}; the proofs of all results in this section are detailed in Section \ref{sec:subspace_LS_analysis}.
\begin{theorem} \label{thm:thm_subs_constr_strict_stab}
Under Assumptions \ref{eq:eta_assump} and \ref{assum:smoothness}, suppose  $\eta_{l,t}$ has $R$-subgaussian entries for each $l,t$. Furthermore, let $G = ([m], \calE)$ be $\theta$-delocalized. Then there exist constants $c_1,c_2, C, C_1, C_2 > 0$ such that the following is true. For $\delta \in (0,e^{-c_1})$, if 
\begin{equation} \label{eq:Tcond_subs_constr_theorem}
    T \geq C R^2\max \set{\max \set{\sqrt{\frac{\theta^3 \tau^3 d T}{m}} , \frac{\theta^2 \tau^2 d}{m}}, \frac{\theta d \tau^2}{m} \left[\log\left(\frac{1}{\delta}\right) + \log\left(1 + \frac{\theta\tau \gamma_2(m,T)}{mT} \right) \right]  },
\end{equation}
it holds w.p at least $1-2\exp(-c_2 \tau d) - 4 \delta$ that the estimator in \eqref{eq:proj_est_closed_form} satisfies
\begin{align} \label{eq:subs_contr_mainbd}
  \frac{1}{m}\sum_{l=1}^m \norm{\est{A}_l - A_l^*}_F^2 \leq C_1\frac{S_m}{m \lambda_{m-\tau}} \left[ 1 + \frac{\gamma_1^2(T,m,\delta)}{T^2} \right] \indic_{\set{\tau \leq m-1}} 
  + C_2 \frac{R^2 d^2 \tau}{Tm} \log\left(\frac{\theta \tau \gamma_3(m,T)}{\rev{\delta} m T} + \frac{1}{\delta} \right).
\end{align}
Here, $\gamma_1(T,m,\delta)$ is as in \eqref{eq:gamma_def_laplsmooth_strict_stab}, while 
\begin{align} \label{eq:gamma_conds_sub_theorem}
   \gamma_2(m,T) := \sum_{l=1}^m \sum_{t=1}^{T} \Tr(\Gamma_t(A_l^*) - I) \quad \text{ and } \quad \gamma_3(m,T) := \sum_{l=1}^m \sum_{t=0}^{T-1} \Tr(\Gamma_t(A_l^*)). 
\end{align}
\end{theorem}
As in Theorem \ref{thm:lapl_smooth_strict_stab}, the proof of Theorem \ref{thm:thm_subs_constr_strict_stab} is also motivated by ideas in the proof of \cite[Theorem 1]{Sarkar19} for learning marginally stable linear dynamical systems. The first error term in \eqref{eq:subs_contr_mainbd} is the bias, while the second term is the variance. The condition in \eqref{eq:Tcond_subs_constr_theorem} arises in order to ensure a lower bound on the smallest eigenvalue of $\Vtau^\top Q^\top Q \Vtau$, see Lemma \ref{lem:eigval_control_mat_subspace_method}. The following points are important to note about this condition when $\theta$ is a constant.
\begin{enumerate}
    \item On one hand, \eqref{eq:Tcond_subs_constr_theorem} imposes very mild requirements on $T$ when $\tau = O(m^{1/3})$, as the RHS of \eqref{eq:Tcond_subs_constr_theorem} can be controlled, for large enough $m$.

    \item On the other hand, if $\tau = \omega(m^{1/3})$, then the RHS of \eqref{eq:Tcond_subs_constr_theorem} increases with $m$, thus imposing stronger requirements on $T$. In particular, $T$ would have to grow with $m$, and depending on the growth of $\tau$ with $m$, could be much larger than the requirement in Theorem \ref{thm:lapl_smooth_strict_stab}.
\end{enumerate}
So the interesting choice of $\tau$ in Theorem \ref{thm:thm_subs_constr_strict_stab} will be $\tau = O(m^{1/3})$. We believe the condition \eqref{eq:Tcond_subs_constr_theorem} can be improved in terms of the dependency on $\tau$, and the RHS should ideally scale linearly with $\tau/m$. Unfortunately, we do not know of a way to improve \eqref{eq:Tcond_subs_constr_theorem} currently\footnote{On the technical front, this condition arises due to the concentration bounds in Proposition's \ref{prop:conc_excit_subsp_constr} and  \ref{prop:spec_norm_bd_subsp_constr}.}. As we will see below for the case where $G$ is a path graph (for which $\theta = 2$ holds), this restriction on $\tau$ will also limit the applicability of Theorem \ref{thm:thm_subs_constr_strict_stab} to the smoothness regime $S_m = o(m^{-1/3})$.

\begin{corollary}[Path graph] \label{corr:subs_constr_strict_stab_path}
    Under Assumptions \ref{eq:eta_assump} and \ref{assum:smoothness}, suppose  $\eta_{l,t}$ has $R$-subgaussian entries for each $l,t$. Furthermore, let $G = ([m], \calE)$ be the path graph. Then there exist constants $c_i, C_i$ ($i=1,2,3$), $C, C', C_3' > 0$ such that the following is true. 
    
    \begin{enumerate}
    \item For $\delta \in (0,e^{-c_1})$, suppose 
    \begin{align} \label{eq:m_cond_corr_subs_constr}
        m \geq C\max\set{ \frac{1}{R^2 \xi^{1/2}(m,T,\delta)}\sqrt{\frac{T}{d}},  \frac{dR^4}{T} \xi(m,T,\delta)}
    \end{align}
    where $\xi(m,T,\delta) := \log(\frac{1}{\delta}) + \log(1 + \frac{\gamma_2(m,T)}{T})$, with $\gamma_2(m,T)$ as in \eqref{eq:gamma_conds_sub_theorem}. Then for the choice $\tau^* = \lfloor \tilde{\tau}^* \rfloor$ where
    \begin{align} \label{eq:tau_tilde_exp_subs_cor}
        \tilde{\tau}^* = \min\set{\left( \frac{m T}{C' dR^4 \xi(m,T,\delta)} \right)^{1/3} , \ \max\set{\Big(\frac{2 m^2 S_m T}{R^2 d^2}\Big)^{1/3},1}},
    \end{align}
    it holds w.p at least  $1-2\exp(-c_2 \tau^* d) - 4 \delta$ that the estimator in \eqref{eq:proj_est_closed_form} satisfies 
    \begin{align} \label{eq:subs_contr_mainbd_path}
    \frac{1}{m}\sum_{l=1}^m \norm{\est{A}_l - A_l^*}_F^2 \leq C_1\frac{m S_m}{(\tau^*)^2} \left[ 1 + \frac{\gamma_1^2(T,m,\delta)}{T^2} \right] \indic_{\set{\tau^* \leq m-1}} 
    + C_2 \frac{R^2 d^2 \tau^*}{Tm} \log\left(\frac{\tau^* \gamma_3(m,T)}{\rev{\delta} m T} + \frac{1}{\delta} \right).
\end{align}

  \item Suppose each $A^*_l$ is marginally stable, i.e. satisfies \eqref{eq:stable_mat_specrad}, and also $T \geq 2$. 
 If $m$ satisfies
 \begin{equation} \label{eq:m_con_subs_cor_stab}
     m \geq C_3 \max\set{\frac{dR^4}{T} \log\Big(\frac{dm T^{\alpha} \kappa^2_{\max}}{\delta}\Big), \frac{1}{R^2 \log^{1/2}(\frac{1}{\delta})}\sqrt{\frac{T}{d}}},
 \end{equation}
then \eqref{eq:m_cond_corr_subs_constr} is ensured. Moreover, if $S_m \geq \frac{d}{c_3 mR^2}$, then for $$\rev{\tau^*} = \left\lfloor\left(\frac{m T}{C_3' dR^4 \log\Big(\frac{dmT^{\alpha} \kappa^2_{\max}}{\delta}\Big)}\right)^{1/3} \right\rfloor$$
the bound in \eqref{eq:subs_contr_mainbd_path} simplifies to
 \begin{equation} \label{eq:subs_constr_stab_path_err_bd}
       \frac{1}{m}\sum_{l=1}^m \norm{\est{A}_l - A_l^*}_F^2 \leq C_4 R^{20/3} \rev{d^{8/3}} \kappa^{4}_{\max} T^{2\alpha-\frac{2}{3}} \log^{2}(\frac{m}{\delta})\Big(m^{1/3} S_m + \rev{m^{-2/3}}\Big) \log^{2/3}\Big(\frac{dmT^{\alpha} \kappa^2_{\max}}{\delta}\Big).
 \end{equation}
\end{enumerate}
\end{corollary}
%
The following remarks are in order.

\begin{enumerate}
\item \rev{The bound in \eqref{eq:subs_constr_stab_path_err_bd} increases with $T$, which is essentially due to the $\gamma_1(T,m,\delta)$ term in \eqref{eq:subs_contr_mainbd_path}. This is an artefact of the analysis, and it is unclear how to improve this for the case of marginally stable matrices. But Remark \ref{rem:exp_dep_T} applies here as well -- in case $\norm{A^*_l}_2 < 1$ for each $l$, we can simply set $\alpha = 0$ and $\kappa^2_{\max} = (1-\max_l \norm{A^*_l}_2^2)^{-1}$ in Corollary \ref{corr:subs_constr_strict_stab_path}. Then, the bound on the MSE will be a decreasing function of $T$.}

\item The second part of the above corollary is written for the regime $S_m = \Omega(1/m)$, but it is not hard to derive error rates for the setting where $S_m = O(1/m)$. Notice that in order to ensure $\frac{1}{m} \sum_{l=1}^m \norm{\est{A}_l - A_l^*}_F^2 = o(1)$ as $m \rightarrow \infty$, we require $S_m = o(m^{-1/3})$, which, as alluded to earlier, is a consequence of the condition \eqref{eq:Tcond_subs_constr_theorem}. To illustrate this smoothness value, recall Example \ref{ex:path_graph_holder} -- in that setting, this would mean that each $A^*_{l,i,j}$ is formed by sampling a \rev{$(M,\beta)$}-H\"older function with $\beta \in (2/3, 1]$.

\item Notice that we only require $T \geq 2$ which is very mild.  In case $T = 2$, and assuming $\kappa_{\max}, R$ to be constants, \eqref{eq:m_con_subs_cor_stab} translates to $m = \Omega(d \log(dm 2^{\alpha}/\delta))$, which leads to the bound 
$$\frac{1}{m}\sum_{l=1}^m \norm{\est{A}_l - A_l^*}_F^2 = O\Big(4^{\alpha} d^{11/3} \log^{8/3}\Big(\frac{dm 2^{\alpha}}{\delta} \Big) (m^{1/3} S_m + \rev{m^{-2/3}})\Big).$$
For instance, if $S_m = O(m^{-2/3})$, this implies that the MSE is $O(4^{\alpha} d^{11/3} \log^{8/3}(\frac{dm 2^{\alpha}}{\delta}) m^{-1/3})$. If $\alpha \asymp d$, then $m$ is required to grow exponentially with $d$ for controlling the MSE. 
\end{enumerate}

%% file: analysis_laplacian_smoothing.tex
%
\section{Analysis: Laplacian smoothing} \label{sec:laplacian_smooth_analysis}
Forming $\est{a} \in \matR^{md^2}$ by column-stacking of $\est{a}_l$'s, the solution of \eqref{eq:vec_pen_ls_algo} satisfies
\begin{equation} \label{eq:pen_ls_lin_sys}
 \left[Q^\top Q + \lambda (L \otimes I_{d^2}) \right] \est{a} =  Q^\top \xtil.
\end{equation}
We need to establish conditions under which $Q^\top Q + \lambda (L \otimes I_{d^2}) \succ 0$, so that \eqref{eq:pen_ls_lin_sys} has a unique solution. In general, this is difficult as the eigenspaces of $Q^\top Q$ and $L \otimes I_{d^2}$ are not aligned.  In Theorem \ref{thm:lapl_smooth_strict_stab}, we will establish the stronger requirement that $Q^\top Q \succ 0$.

%
\subsection{Proof of Theorem \ref{thm:lapl_smooth_strict_stab}} \label{subsec:proof_thm_lapl_smooth}
Let us denote  
\begin{equation*}
   \Mlam = Q^\top Q + \lambda (L \otimes I_{d^2}) \quad \text{and} \quad  E_l = [\eta_{l,2} \ \eta_{l,3} \ \cdots \ \eta_{l,T+1}]
\end{equation*}
with $\eta_l = \vect(E_l)$. Form $\eta \in \matR^{mdT}$ by column-wise stacking $\eta_1,\dots,\eta_m$. If $\Mlam \succ 0$, then \eqref{eq:pen_ls_lin_sys} implies
\begin{align*}
    \est{a} = \Mlam^{-1} Q^\top (\underbrace{Q a^* + \eta}_{=\xtil}) 
    = a^* - \underbrace{\lambda \Mlam^{-1} (L \otimes I_{d^2}) a^*}_{=: \ e_1 (\text{Bias})} + \underbrace{\Mlam^{-1} Q^{\top} \eta}_{=: \ e_2 (\text{Variance})}.
\end{align*}
Hence $\Mlam \succ 0$ implies $\norm{\est{a} - a^*}_2 \leq \norm{e_1}_2 + \norm{e_2}_2$, and so, 
\begin{equation} \label{eq:lapsmooth_temp_1}
   \sum_{l=1}^m \norm{\est{A}_l - A_l^*}_F^2 = \norm{\est{a} - a^*}_2^2 \leq 2(\norm{e_1}_2^2 + \norm{e_2}_2^2).
\end{equation}
In order to ensure $\Mlam \succ 0$, we will use a stronger condition $Q^\top Q \succ 0$ (for the reasons mentioned earlier). In fact, we will need a sufficiently large lower bound on the smallest eigenvalue of $Q^{\top} Q$.
%
%
\begin{lemma} \label{lem:laplace_smooth_qq_lowbd}
Denote the event 
\begin{equation*}
    \calE_1 = \set{Q^\top Q \succeq \frac{T}{4} (I_m \otimes I_{d^2})}.
\end{equation*}
If $T$ satisfies  \eqref{eq:T_cond_Lapsmooth_strict_stab} for any $\delta \in (0,1)$, then $\prob(\calE_1) \geq 1-\delta$. 
\end{lemma}
\begin{proof}
    Recall $Q = \blkdiag(X_l^\top \otimes I_d)_{l=1}^m$. For any given $l$, it was shown in \citep[Section 9]{Sarkar19} that $X_l X_l^\top \succeq (T/4) I_{d}$ w.p at least $1- \delta$ provided
    $T$ satisfies \eqref{eq:T_cond_Lapsmooth_strict_stab} with $m$ replaced by $1$. Taking a union bound over $[m]$ then completes the proof.
\end{proof}
It remains to bound $\norm{e_1}_2, \norm{e_2}_2$, which upon plugging in \eqref{eq:lapsmooth_temp_1} leads to the statement of the theorem.  
%
\begin{lemma}[Bound on $\norm{e_1}_2$] \label{lem:bd_laplsmooth_e_1}
\begin{align*}
   \calE_1 \implies \norm{e_1}_2^2 \leq \frac{16 \lambda^2}{T^2} \norm{(L \otimes I_{d^2}) a^*}_2^2.
\end{align*}
\end{lemma}
\begin{proof}
    Notice that $\calE_1 \implies \Mlam \succeq (\frac{T}{4} I_m + \lambda L) \otimes I_{d^2}$, which in turn implies\footnote{Recall that for positive definite matrices $X, Y$, it holds that $X \succ Y \iff X^{-1} \prec Y^{-1}$.}
    \begin{align*}
         \Mlam^{-1} \preceq \left[\frac{4}{T} v_m v_m^\top + \sum_{l=1}^{m-1} \Big(\frac{T}{4} + \lambda \lambda_l \Big)^{-1} v_l v_l^\top \right] \otimes I_{d^2}.
    \end{align*}
    In particular, $\norm{\Mlam^{-1}}_2 \leq 4/T$ which using $\norm{e_1}_2 \leq \lambda \norm{\Mlam^{-1}}_2 \norm{(L \otimes I_{d^2}) a^*}_2$ leads to the stated bound. 
\end{proof}

\begin{lemma}[Bound on $\norm{e_2}_2$] \label{lem:bd_laplsmooth_e_2}
There exist constants $c, C > 0$ such that for any $\delta \in (0,e^{-c})$, it holds w.p at least $1-3\delta$ that 
\begin{align*}
    \norm{e_2}_2^2 
    &\leq C\frac{R^2}{T}\left( \log(1/\delta) + d^2 \log \Big(1 + \frac{4\gamma_1(T,m,\delta)}{T} \Big) + d^2 \sum_{l=1}^{m-1} \log\Big(1 + \frac{4\gamma_1(T,m,\delta)}{4\lambda \lambda_l + T}\Big) \right) \\
&\leq C\frac{R^2}{T}\left( \log(1/\delta) + d^2 \log \Big(1 + \frac{4\gamma_1(T,m,\delta)}{T} \Big) + 4 d^2\gamma_1(T,m,\delta) \sum_{l=1}^{m-1}   \frac{1}{4\lambda \lambda_l + T} \right)
\end{align*}
for $\gamma_1(T,m,\delta)$ as in \eqref{eq:gamma_def_laplsmooth_strict_stab}.
\end{lemma}
\begin{proof}
   $\calE_1$ implies 
    \begin{equation*}
      \norm{e_2}_2^2 = \norm{\Mlam^{-1} Q^\top \eta}_2^2 \leq \norm{\Mlam^{-1/2}}_2^2 \norm{\Mlam^{-1/2} Q^\top \eta}_2^2 \leq \frac{4 \norm{Q^\top \eta}_{\Mlam^{-1}}^2}{T}  
    \end{equation*}
    where the final bound uses the fact $\norm{\Mlam^{-1/2}}_2 \leq 2/\sqrt{T}$, as seen from the proof of Lemma \ref{lem:bd_laplsmooth_e_1}. Now note that $Q^\top Q \succeq \frac{T}{4} (I_m \otimes I_{d^2})$ is equivalent to
    \begin{align*}
         2Q^\top Q + \lambda (L \otimes I_{d^2}) &\succeq \frac{T}{4} (I_m \otimes I_{d^2}) + Q^\top Q + \lambda (L \otimes I_{d^2}) \\
         \implies \Mlam^{-1} &\preceq 2\Big(Q^\top Q + (\lambda L + \frac{T}{4} I_m) \otimes I_{d^2} \Big)^{-1}.  
    \end{align*}
    Denoting $\Mlamtil := Q^\top Q + (\lambda L + \frac{T}{4} I_m) \otimes I_{d^2}$, this means
    \begin{align*}
        \norm{Q^\top \eta}_{\Mlam^{-1}}^2 \leq 2 \norm{Q^\top \eta}_{\Mlamtil^{-1}}^2.
    \end{align*}
To complete the proof, we need the following bound on $\norm{Q^\top \eta}_{\Mlamtil^{-1}}^2$.
\begin{proposition} \label{prop:bd_self_norm_laplsmooth}
    For and $\delta \in (0,1)$ consider the event
    \begin{equation*}
        \calE_2 := \set{\norm{Q^\top \eta}_{\Mlamtil^{-1}}^2 \leq 2R^2\left( \log(1/\delta) + \frac{d^2}{2} \log \Big(1 + \frac{4\gamma_1(T,m,\delta)}{T} \Big) + \frac{d^2}{2} \sum_{l=1}^{m-1} \log\Big(1 + \frac{4\gamma_1(T,m,\delta)}{4\lambda \lambda_l + T}\Big) \right) }.
    \end{equation*}
    Then there exists a constant $c > 0$ such that for any $\delta \in (0,e^{-c})$, $\prob(\calE_2) \geq 1-2\delta$.
\end{proposition}
The proof of Proposition \ref{prop:bd_self_norm_laplsmooth} is along the same lines as that of \cite[Theorem 1]{abbasi11} where a  concentration bound for ``self-normalized'' vector-valued  martingales was obtained. We defer it to Appendix \ref{appsec:proofs_lapl_smooth}.

The statement of Lemma \ref{lem:bd_laplsmooth_e_2} now follows on the event $\calE_1 \cap \calE_2$. Note that the second inequality therein is a consequence of the fact $\log (1 + x) \leq x$ for all $x > -1$.
\end{proof}
Finally, we note that the statement of Theorem \ref{thm:lapl_smooth_strict_stab} follows directly from Lemma's \ref{lem:bd_laplsmooth_e_1} and \ref{lem:bd_laplsmooth_e_2}, on the event $\calE_1 \cap \calE_2$.

%
\subsection{Proof of Corollary \ref{corr:lapl_smooth_strict_stab}} \label{subsec:proof_corr_laplsmooth_path}
Recall that for a path graph, the eigenvalues of $L$ are given by (e.g., \citep{brouwer12})
\begin{align*}
   \lambda_l = 4\sin^2\Big(\frac{\pi}{2m}(m-l)\Big); \quad l=1,\dots,m.
\end{align*}
Since $\lambda_l^2 \leq 4 \lambda_l$ for each $l$, we then obtain 
\begin{align*}
    \norm{(L \otimes I_{d^2}) a^*}_2^2 = \sum_{l=1}^m \sum_{i=1}^{d^2} \lambda_l^2 \dotprod{a^*}{v_l \otimes e_i}^2 \leq 4 \sum_{l=1}^m \sum_{i=1}^{d^2} \lambda_l \dotprod{a^*}{v_l \otimes e_i}^2 = (a^*)^{\top} (L \otimes I_{d^2}) a^*.
\end{align*}
Hence $\norm{(L \otimes I_{d^2}) a^*}_2^2 \leq S_m$ under Assumption \ref{assum:smoothness}. 

Next, using the fact $\sin x \in [\frac{2x}{\pi}, x]$ for $x \in [0,\pi/2]$, this implies $\lambda_l \in [c_1 \frac{(m-l)^2}{m^2}, c_2 \frac{(m-l)^2}{m^2}]$ for each $l$. Hence we can bound $ \sum_{l=1}^{m-1} \frac{1}{4\lambda \lambda_l + T}$ as 
\begin{align} \label{eq:temp_riemann_sum}
     \sum_{l=1}^{m-1} \frac{1}{4\lambda \lambda_l + T} \leq \sum_{l=1}^{m-1} \frac{1}{c_1 \lambda (1-\frac{l}{m})^2 + T}.   
\end{align}
Denoting $f(x) = \frac{1}{c_1 \lambda (1-x)^2 + T}$ for $x \in [0,1]$, note that $f$ is \rev{increasing} for $x \in [0,1]$. Hence by recognizing the RHS of \eqref{eq:temp_riemann_sum} as a ($m$ times the) Riemannian sum for $\int_{0}^1 f(x) dx$, we obtain
\begin{align*}
    \sum_{l=1}^{m-1} \frac{1}{4\lambda \lambda_l + T} \leq m \int_{0}^1 \frac{1}{c_1 \lambda (1-x)^2 + T} dx = m \int_{0}^1 \frac{1}{c_1 \lambda y^2 + T} d y = \frac{m}{\sqrt{c_1 \lambda T}} \arctan\Big(\sqrt{\frac{c_1 \lambda}{T}}\Big).
\end{align*}

Applying the above bounds to \eqref{eq:lapl_smooth_mainbd}, we then obtain  (after some minor simplifications)
\begin{align} \label{eq:lapl_smooth_path_graph_mainbd}
    \frac{1}{m} \sum_{l=1}^m \norm{\est{A}_l - A_l^*}_F^2 
    &\leq 
    \frac{C_1}{m}\Bigg(\frac{\lambda^2 S_m}{T^2}  + \frac{R^2}{T} \left[ \log\left(\frac{1}{\delta}\right) + d^2 \log\left( 1 + \frac{\gamma_1 (T,m,\delta)}{T}\right) \right] \nonumber \\
    &+ \frac{R^2 d^2 m \gamma_1 (T,m,\delta)}{T^{3/2} \sqrt{\lambda}} \arctan \left(\sqrt{\frac{C_2 \lambda}{T}} \right) \Bigg).
\end{align}

To show the first part of the corollary, we use the fact $\arctan(\cdot) \leq \pi/2$, and observe that the function $f(\lambda) = \frac{\lambda^2 S_m}{T^2} + (\frac{R^2 d^2}{T^{3/2}}) \frac{m}{\sqrt{\lambda}}$ is convex for $\lambda > 0$ and is minimized for $\lambda = (Rd)^{4/5} (\frac{m}{S_m})^{2/5} T^{1/5}$. Plugging this value in \eqref{eq:lapl_smooth_path_graph_mainbd} and observing that $\gamma_1(T,m,\delta) > 1$ leads to \eqref{eq:lapl_smooth_path_graph_mainbd_1} after some minor simplification.

For the second part, use \eqref{eq:trace_gamma_bd_stabmat} and the monotonicity of $\Tr (\Gamma_t(A_l^*))$ w.r.t $t$ to obtain
\begin{equation} \label{eq:trace_bd_stable_temp}
    \sum_{t=0}^{T-1} \Tr (\Gamma_t(A_l^*)) \lesssim d T^{\alpha_l + 1} \kappa^2(P_l)
\end{equation}
for $\alpha_l \in [d]$ as in \eqref{eq:trace_gamma_bd_stabmat}. Recalling \eqref{eq:alpha_kappa_defs_stabmat}, this implies  
\begin{equation} \label{eq:gamma_1_bd_stable}
\gamma_1(T,m,\delta) \lesssim d R^2 \log(m/\delta) T^{\alpha+1} \kappa^2_{\max}.
\end{equation}
Using \eqref{eq:trace_bd_stable_temp} and \eqref{eq:gamma_1_bd_stable} in \eqref{eq:T_cond_Lapsmooth_strict_stab} and \eqref{eq:lapl_smooth_path_graph_mainbd_1} respectively, we obtain the second part of the corollary after some minor simplifications (noting that $R \gtrsim 1$ due to Assumption \ref{eq:eta_assump}).

%
\subsection{Proof of Corollary \ref{corr:lapl_smooth_strict_stab_compl_star}}
\subsubsection{Proof for complete graph}
Recall that for a complete graph, the eigenvalues of $L$ are given by \cite{brouwer12}
\begin{align*}
 \lambda_m = 0, \text{ and } \lambda_l = m \text{ for } l=1,\dots,m-1.
\end{align*}
This leads to the bound
\begin{align*}
    \norm{(L \otimes I_{d^2}) a^*}_2^2 = \sum_{l=1}^m \sum_{i=1}^{d^2} \lambda_l^2 \dotprod{a^*}{v_l \otimes e_i}^2 = m^2 \sum_{l=1}^{m-1} \sum_{i=1}^{d^2} \dotprod{a^*}{v_l \otimes e_i}^2 = m (a^*)^{\top} (L \otimes I_{d^2}) a^*.
\end{align*}
Hence $\norm{(L \otimes I_{d^2}) a^*}_2^2 \leq m S_m$ under Assumption \ref{assum:smoothness}. Moreover, $\sum_{l=1}^{m-1} \frac{1}{4\lambda \lambda_l + T} \leq \frac{m}{m\lambda + T}$. Using these bounds in \eqref{eq:lapl_smooth_mainbd} and the fact $\gamma_1(T,m,\delta) \geq 1$, we obtain 
\begin{align} 
    &\frac{1}{m} \sum_{l=1}^m \norm{\est{A}_l - A_l^*}_F^2 \nonumber \\
    &\lesssim  \frac{\lambda^2 S_m}{T^2}  + \frac{R^2}{m T} \left[ \log\left(\frac{1}{\delta}\right) + d^2 \log\left( 1 + \frac{\gamma_1 (T,m,\delta)}{T}\right) \right]  
    + \frac{R^2 d^2 \gamma_1 (T,m,\delta)}{T} \left[ \frac{1}{m\lambda + T} \right] \nonumber \\ 
    &\leq \left(\frac{\lambda^2 S_m}{T^2} + \frac{R^2 d^2}{T}\frac{1}{m\lambda}  \right) \gamma_1 (T,m,\delta) + \frac{R^2}{m T} \left[ \log\left(\frac{1}{\delta}\right) + d^2 \log\left( 1 + \frac{\gamma_1 (T,m,\delta)}{T}\right) \right]. \label{eq:corr_lapl_smooth_compl_temp1}
\end{align}
Now notice that $\lambda^* = (\frac{T R^2 d^2}{m S_m})^{1/3}$ minimizes the RHS of \eqref{eq:corr_lapl_smooth_compl_temp1} (up to a constant). Plugging $\lambda = \lambda^*$ in \eqref{eq:corr_lapl_smooth_compl_temp1} leads to the bound
\begin{align} \label{eq:corr_lapl_smooth_compl_temp2}
   \frac{1}{m} \sum_{l=1}^m \norm{\est{A}_l - A_l^*}_F^2  
    &\lesssim   
   \Big(\frac{S_m^{1/3}}{m^{2/3}} \Big) \Big(\frac{Rd}{T} \Big)^{4/3} \gamma_1(T,m,\delta) \nonumber \\ 
    &+ \frac{R^2}{m T} \left[ \log\left(\frac{1}{\delta}\right) + d^2 \log\left( 1 + \frac{\gamma_1 (T,m,\delta)}{T}\right) \right].
\end{align}
Finally we obtain the stated error bound upon using  \eqref{eq:gamma_1_bd_stable} in \eqref{eq:corr_lapl_smooth_compl_temp2}.

%

\subsubsection{Proof for star graph}
Recall that for a star graph, the eigenvalues of $L$ are given by \cite{spielmanSGT}
\begin{align*}
 \lambda_m = 0, \ \lambda_1 = m, \ \text{ and } \lambda_l = 1 \text{ for } l=2,\dots,m-1 
\end{align*}
\rev{and also $v_1 = [m(m-1)]^{-1/2} (-(m-1),1,\cdots,1)^\top$}. This leads to the bound
\begin{align*}
    \norm{(L \otimes I_{d^2}) a^*}_2^2 
    &= \sum_{l=1}^m \sum_{i=1}^{d^2} \lambda_l^2 \dotprod{a^*}{v_l \otimes e_i}^2 \\
    &= m^2  \sum_{i=1}^{d^2} \dotprod{a^*}{v_1 \otimes e_i}^2 + \sum_{l=2}^{m-1} \sum_{i=1}^{d^2}  \dotprod{a^*}{v_l \otimes e_i}^2 \\
    &\leq \rev{(m^2 - m) \sum_{i=1}^{d^2} \dotprod{a^*}{v_1 \otimes e_i}^2 + S_m} \\
    &\leq \rev{m(m-1) (a^*)^\top [(v_1 \otimes I_{d^2}) (v_1^{\top} \otimes I_{d^2})] a^* + S_m} \\
    &\leq \rev{\tilde{S}_m}, 
\end{align*}
\rev{where the last equality follows from the expression for $v_1$ with some minor simplification.} It is also easy to verify that $\sum_{l=1}^{m-1} \frac{1}{4\lambda \lambda_l + T} \leq \frac{2m}{\lambda + T}$.  Using these bounds in \eqref{eq:lapl_smooth_mainbd} and the fact $\gamma_1(T,m,\delta) \geq 1$, we obtain 
\begin{align} 
    &\frac{1}{m} \sum_{l=1}^m \norm{\est{A}_l - A_l^*}_F^2 \nonumber \\
    &\lesssim  \frac{\lambda^2 \tilde{S}_m}{T^2}  + \frac{R^2}{m T} \left[ \log\left(\frac{1}{\delta}\right) + d^2 \log\left( 1 + \frac{\gamma_1 (T,m,\delta)}{T}\right) \right]  
    + \frac{R^2 d^2 \gamma_1 (T,m,\delta)}{T} \left[ \frac{1}{\lambda + T} \right] \nonumber \\ 
    &\leq \left(\frac{\lambda^2 \tilde{S}_m}{T^2} + \frac{R^2 d^2}{T}\frac{1}{\lambda}  \right) \gamma_1 (T,m,\delta) + \frac{R^2}{m T} \left[ \log\left(\frac{1}{\delta}\right) + d^2 \log\left( 1 + \frac{\gamma_1 (T,m,\delta)}{T}\right) \right]. \label{eq:corr_lapl_smooth_star_temp1}
\end{align}
The RHS of \eqref{eq:corr_lapl_smooth_star_temp1} is minimized for $\lambda =  (\frac{T R^2 d^2 m}{ \tilde{S}_m})^{1/3}$ (up to a constant). Plugging $\lambda^*$ in \eqref{eq:corr_lapl_smooth_star_temp1} leads to the bound 
\begin{align} \label{eq:corr_lapl_smooth_star_temp2}
 \frac{1}{m} \sum_{l=1}^m \norm{\est{A}_l - A_l^*}_F^2 
  &\lesssim \Big(\frac{\tilde{S}_m}{m}\Big)^{1/3}  (\frac{Rd}{T})^{4/3} \gamma_1(T,m,\delta) \nonumber \\ 
    &+ \frac{R^2}{m T} \left[ \log\left(\frac{1}{\delta}\right) + d^2 \log\left( 1 + \frac{\gamma_1 (T,m,\delta)}{T}\right) \right].
\end{align}
Finally we obtain the stated error bound upon using  \eqref{eq:gamma_1_bd_stable} in \eqref{eq:corr_lapl_smooth_star_temp2}.

%% file: analysis_subspace_constraint.tex
\section{Analysis: Subspace constrained LS} \label{sec:subspace_LS_analysis}
Denote $\Mtau = \Ptilproj Q^\top Q \Ptilproj $ for convenience. Also, let $\Ptilproj_{\perp}$ denote the projection matrix for the space orthogonal to $\text{span}(\Ptilproj)$. Using $\xtil = Qa^* + \eta$, where we recall $\eta \in \matR^{dmT}$ from Section \ref{subsec:proof_thm_lapl_smooth}, we can then write $\est{a}$ as 
\begin{align*}
\est{a} &= \Mtau^\dagger  \Ptilproj Q^\top (Qa^* + \eta) \\ 
&=  \Mtau^\dagger  \Ptilproj Q^\top Q \Big(\Ptilproj a^* + \Ptilproj_{\perp} a^* \Big) + \Mtau^\dagger \Ptilproj Q^\top \eta \\
&= \Mtau^\dagger \Mtau a^* + \Mtau^\dagger \Ptilproj Q^\top Q \Ptilproj_{\perp} a^* + \Mtau^\dagger \Ptilproj Q^\top \eta, 
\end{align*}
which implies
\begin{align} \label{eq:err_decomp_subsp_const}
    \est{a} - a^* = \underbrace{-(I_{md^2} - \Mtau^\dagger \Mtau)) a^*}_{=: e_1 \ (\text{Bias})} + \underbrace{\Mtau^{\dagger} \Ptilproj Q^\top Q \Ptilproj_{\perp} a^*}_{=: e_2 \ \text{(Bias)}} + \underbrace{\Mtau^\dagger \Ptilproj Q^\top \eta}_{=: e_3 \ \text{(Variance)}}.
\end{align}
The terms $e_1,e_2$ are the bias terms which we expect to decrease with $\tau$. On the other hand, $e_3$ is the variance term which will increase with $\tau$. 

%
\subsection{Proof of Theorem \ref{thm:thm_subs_constr_strict_stab}}
Before proceeding, it will be useful to denote \rev{an orthonormal-basis matrix for $\text{span}(\Ptilproj)$ by $\Vtau \in \matR^{md^2 \times \tau d^2}$}. In fact given the definition of $\Ptilproj$, we can explicitly write $\Vtau$ as 
\begin{align} \label{eq:Vtau_def}
    \Vtau = [v_m \ v_{m-1} \ \cdots \ v_{m-\tau+1}] \otimes I_{d^2}
\end{align}
where we recall from Section \ref{subsec:setup} that $v_l \in \matR^m$ is the eigenvector (associated with eigenvalue $\lambda_l$) of the Laplacian $L$. 

The key quantity in our analysis will be the matrix $\Vtau^\top Q^\top Q \Vtau$. Notice that $M_{\tau} = \Vtau \Vtau^\top Q^\top Q \Vtau \Vtau^\top$, hence if $\Vtau^\top Q^\top Q \Vtau \succ 0$, then it implies 
\begin{equation} \label{eq:subs_analysis_Mtau_dag_proj}
    M_{\tau}^{\dagger} M_{\tau} = \Big[\Vtau \Vtau^\top Q^\top Q \Vtau \Vtau^\top \Big]^{\dagger} \Big(\Vtau \Vtau^\top Q^\top Q \Vtau \Vtau^\top \Big) = \Vtau \Vtau^\top \ ( = \Ptilproj).
\end{equation}

If $\Vtau^\top Q^\top Q \Vtau \succ 0$, then the term $e_1$ is easily bounded as follows. 
\begin{lemma}[Bound $\norm{e_1}_2$] \label{lem:subsp_e_1}
   Recall the notation for the spectrum of $L$ from Section \ref{sec:prob_setup}. If Assumption  \ref{assum:smoothness} holds, then for any graph $G$, 
    \begin{align*}
    \Vtau^\top Q^\top Q \Vtau \succ 0 \implies \norm{e_1}_2^2 \leq C \frac{S_m}{\lambda_{m-\tau}}\indic_{\tau \leq m-1} 
    \end{align*}
    for a constant $ C > 0$.
\end{lemma}
\begin{proof}
If $\Vtau^\top Q^\top Q \Vtau \succ 0$, then using \eqref{eq:subs_analysis_Mtau_dag_proj} we obtain $e_1 = -\Ptilproj_{\perp} a^*$, and then bound $\norm{e_1}_2$ as
    \begin{align*}
        \norm{e_1}_2^2 = \norm{\Ptilproj_{\perp} a^*}_2^2 \leq \frac{a^* (L \otimes I_{d^2}) a^*}{\lambda_{m-\tau}} \indic_{\tau \leq m-1} 
        \leq \frac{S_m}{\lambda_{m-\tau}} \indic_{\tau \leq m-1}.
    \end{align*}
\end{proof}

For bounding the terms $\norm{e_2}_2$ and $\norm{e_3}_2$, we will actually require a sufficiently large lower bound on  $\lambda_{\min}(\Vtau^\top Q^\top Q \Vtau)$ (the smallest eigenvalue of $\Vtau^\top Q^\top Q \Vtau$). This is captured by the following crucial lemma.
%
%
%
\begin{lemma} \label{lem:eigval_control_mat_subspace_method}
Let $G$ be $\theta$-delocalized \rev{and denote the event} 
\begin{align} \label{eq:event_eigval_qqmat}
    \calE_3 := \set{\Vtau^\top Q^\top Q \Vtau \succeq \frac{T}{4} I_{\tau d^2}}.
\end{align}  
For any $\delta \in (0,1)$, suppose that $T, \tau$ satisfy
\eqref{eq:Tcond_subs_constr_theorem} for a suitable choice of constant $C > 0$. Then there exists a constant $c > 0$ such that $\prob(\calE_3) \geq 1-2\exp(-c\theta \tau d) - \delta$.
\end{lemma}
The proof is deferred to Appendix \ref{appsec:sub_constr_proofs} -- it is essentially based on adapting the arguments in \cite[Section 9]{Sarkar19} to our setup.

Using Lemma \ref{lem:eigval_control_mat_subspace_method} and Lemma \ref{lem:Qmat_bd}, we can easily bound $\norm{e_2}_2$.
\begin{lemma}[Bound $\norm{e_2}_2$] \label{lem:e_2_bd_sub_constr}
Let $G$ be $\theta$-delocalized and suppose Assumption \ref{assum:smoothness} holds. Recall the events $\calE_3$, $\calE_{2b}$ defined in \eqref{eq:event_eigval_qqmat} and   \eqref{eq:qq_lowner_bd} respectively. Then,  

    \begin{equation*}
    \calE_3 \cap \calE_{2b} \implies   \norm{e_2}_2^2 \leq \frac{\gamma_1^2(T,m,\delta)}{T^2} \norm{e_1}_2^2 \leq C \frac{\gamma_1^2(T,m,\delta)}{T^2} \frac{S_m}{\lambda_{m-\tau}}\indic_{\tau \leq m-1}
    \end{equation*}
    for $\gamma_1(T,m,\delta)$ as in \eqref{eq:gamma_def_laplsmooth_strict_stab}.
\end{lemma}
\begin{proof}
    Using the fact $\Ptilproj = \Vtau \Vtau^\top$, notice that 
    \begin{align} \label{eq:e2_bd_temp_1}
        (\Ptilproj Q^\top Q \Ptilproj)^\dagger = \Big[\Vtau (\Vtau^\top Q^\top Q \Vtau) \Vtau^\top \Big]^{\dagger} = \Vtau (\Vtau^\top Q^\top Q \Vtau)^\dagger \Vtau^\top.
    \end{align}
    Then we can bound $\norm{e_2}_2$ as
    \begin{align*}
        \norm{e_2}_2^2 &= \norm{(\Ptilproj Q^\top Q \Ptilproj)^\dagger \Ptilproj Q^\top Q \Ptilproj_{\perp} a^*}_2 \\
        &\leq \norm{(\Ptilproj Q^\top Q \Ptilproj)^\dagger \Ptilproj Q^\top Q}_2^2  \ \norm{e_1}_2^2 \\
        &\leq \norm{(\Ptilproj Q^\top Q \Ptilproj)^\dagger}_2^2 \ \norm{ \Ptilproj Q^\top Q}_2^2 \ \norm{e_1}_2^2  \\
        &= \norm{(\Vtau^\top Q^\top Q \Vtau)^\dagger}_2^2 \ \norm{\Vtau^\top Q^\top Q}_2^2 \ \norm{e_1}_2^2  \tag{using \eqref{eq:e2_bd_temp_1}} \\
        &\leq \frac{16 \gamma_1^2(T,m,\delta)}{T^2} \ \norm{e_1}_2^2,  \tag{on $\calE_3 \cap \calE_{2b}$} %
    \end{align*}
    and bound $\norm{e_1}_2$ using Lemma \ref{lem:subsp_e_1}.
\end{proof}
It now remains to bound $\norm{e_3}_2$. Notice that $\calE_3$ implies $\Vtau^\top Q^\top Q \Vtau$ is invertible, and so
\begin{align*}
    \norm{e_3}_2 = \norm{\Vtau (\Vtau^\top Q^\top Q \Vtau)^{-1} \Vtau^\top Q^\top \eta}_2 &\leq \frac{\norm{(\Vtau^\top Q^\top Q \Vtau)^{-1/2} \Vtau^\top Q^\top \eta}_2}{\sqrt{\lambmin(\Vtau^\top Q^\top Q \Vtau)}}  \\
    &\leq \frac{2}{\sqrt{T}} \norm{(\Vtau^\top Q^\top Q \Vtau)^{-1/2} \Vtau^\top Q^\top \eta}_2.
\end{align*}
To bound the last term above, note that $\calE_3$ also implies 
\begin{align}
     (\Vtau^\top Q^\top Q \Vtau + (T/4)I_{\tau d^2})^{-1} &\succeq \frac{1}{2} (\Vtau^\top Q^\top Q \Vtau)^{-1} \nonumber \\
    \implies \norm{(\Vtau^\top Q^\top Q \Vtau)^{-1/2}\Vtau^\top Q^\top \eta}_{2} &\leq \sqrt{2} \norm{(\Vtau^\top Q^\top Q \Vtau + (T/4)I_{\tau d^2})^{-1/2} \Vtau^\top Q^\top \eta}_{2}. \label{eq:e2_bd_temp_2}
\end{align}
The RHS of \eqref{eq:e2_bd_temp_2} is bounded as follows.
\begin{proposition} \label{prop:e3_bd_imp_event_subsp}
Let $G$ be $\theta$-delocalized. For any $\delta \in (0,1)$, denote the event 
    \begin{align*}
        \calE_4 := \set{\norm{(\Vtau^\top Q^\top Q \Vtau + (T/4)I_{\tau d^2})^{-1/2} \Vtau^\top Q^\top \eta}_{2}^2 \leq C \tau R^2  d^2 \left[\log\Big(\frac{1}{\delta} \Big) + \log\Big(\frac{\theta \tau \gamma_3(T,m)}{mT}  + 1 \Big) \Big)\right]},
    \end{align*}
    for $\gamma_3(T,m)$ as in \eqref{eq:gamma_conds_sub_theorem}, and a suitably large constant $C > 0$. Then, $\prob(\calE_4) \geq  1-2\delta$.
\end{proposition}
The proof is deferred to Appendix \ref{appsec:sub_constr_proofs}. Based on the preceding discussion, it follows that 
\begin{equation} \label{eq:e3_bd_subsp_fin}
   \calE_3 \cap \calE_4  \implies \norm{e_3}_2^2 \leq C\frac{\tau R^2  d^2}{T} \left[\log\Big(\frac{1}{\delta} \Big) + \log\Big(\frac{\theta \tau \gamma_3(T,m)}{mT}  + 1 \Big) \right].
\end{equation}
The statement of the theorem now follows readily using \eqref{eq:err_decomp_subsp_const}, Lemma's \ref{lem:subsp_e_1}, \ref{lem:eigval_control_mat_subspace_method} and \ref{lem:e_2_bd_sub_constr}, Proposition \ref{prop:e3_bd_imp_event_subsp} and \eqref{eq:e3_bd_subsp_fin}. In particular, the success probability equals $\prob(\calE_3 \cap \calE_{2b} \cap \calE_4) \geq 1 - \prob(\calE_3^c) - \prob(\calE_{2b}^c) - \prob(\calE_4^c)$. 

\subsection{Proof of Corollary \ref{corr:subs_constr_strict_stab_path}}
We first recall that $\lambda_l \geq c_1 (m-l)^2/m^2$ for path graphs (as explained earlier in Section \ref{subsec:proof_corr_laplsmooth_path}), for some constant $c_1 > 0$. Also recall the following expression for the eigenvectors of the path graph \citep{brouwer12}. For each $l=1,\dots,m$
\begin{equation*}
    v_{m,l} = \frac{1}{\sqrt{m}}, \quad v_{m-i,l} = \sqrt{\frac{2}{m}} \cos\Big( \frac{(2l-1)\pi i}{2m} \Big); \quad i=1,\dots,m-1.
\end{equation*}
It is then not difficult to verify that $G$ satisfies \eqref{eq:theta_deloc_cond} with $\theta = 2$.
\subsubsection{Proof of Part $1$}
Since $\tau \leq m$, note that \eqref{eq:Tcond_subs_constr_theorem} is ensured if 
\begin{equation} \label{eq:corr_subs_constr_temp1} 
    T \geq C R^2\max \set{\max \set{\sqrt{\frac{\tau^3 d T}{m}} , \frac{\tau^2 d}{m}}, \frac{d \tau^2}{m} \left[\log\left(\frac{1}{\delta}\right) + \log\left(1 + \frac{\gamma_2(m,T)}{T} \right) \right]  }.
\end{equation}
Now $\sqrt{\frac{\tau^3 d T}{m}} \geq \frac{\tau^2 d}{m}$ is equivalent to $\tau \leq mT/d$. Since $\xi(m,T,\delta) \geq 1$ for a suitably large constant $c_1$, hence if the conditions $\tau \leq mT/d$ and  
\begin{align} \label{eq:corr_subs_constr_temp2} 
\tau \leq \Big(\frac{m T}{C^2 d R^4 \xi(m,T,\delta)} \Big)^{1/3} =: \bar{\tau}^*
\end{align}
are satisfied, then clearly \eqref{eq:corr_subs_constr_temp1} is ensured. This latter condition  in \eqref{eq:corr_subs_constr_temp2} is sufficient (i.e., subsumes $\tau \leq mT / d$) when $m$ is large enough. More precisely, 
\begin{align} \label{eq:corr_subs_constr_temp3}
    m \geq \frac{d}{C R^2 T \xi^{1/2}(m,T,\delta)}.
\end{align}
To summarize, if $m$ satisfies \eqref{eq:corr_subs_constr_temp3}, and $\tau$ satisfies \eqref{eq:corr_subs_constr_temp2}, then the condition \eqref{eq:Tcond_subs_constr_theorem} is ensured.

Next, we need to ensure that $\bar{\tau}^* \in [1,m]$ so that $\lfloor \bar{\tau}^* \rfloor \in [m]$. This is the case provided $m$ additionally satisfies
\begin{align} \label{eq:corr_subs_constr_temp4}
    m \geq \max\set{\frac{dR^4 \xi(m,T,\delta)}{T}, \frac{1}{R^2 \xi^{1/2}(m,T,\delta)} \sqrt{\frac{T}{d}}}.
\end{align}
Since $R\gtrsim 1$ and $\xi(m,T,\delta) \geq 1$ one can verify that 
the condition on $m$ in \eqref{eq:m_cond_corr_subs_constr} subsumes \eqref{eq:corr_subs_constr_temp3} and \eqref{eq:corr_subs_constr_temp4}. 
Now consider the function $f(x) = \frac{m S_m}{x^2} + \frac{R^2 d^2 x}{T m}$ as a `surrogate' for the error bound in Theorem \ref{thm:thm_subs_constr_strict_stab}, for $x \in \matR^+$, and note that it is minimized (over $\matR^+$) at $x^* = (\frac{2 m^2 S_m T}{R^2 d^2})^{1/3}$. Hence it follows that 
\begin{equation*}
    \tilde{\tau}^* = \argmin{1 \leq x \leq \bar{\tau}^*} \  f(x) = \min\set{\bar{\tau}^*, \max\set{x^*,1}},
\end{equation*}
and we see that $\tau^* = \lfloor \tilde{\tau}^* \rfloor \in [m]$. This completes the proof of the first part of the corollary.

\subsubsection{Proof of Part $2$}
To show the second part, let us bound $\gamma_2(m,T)$ and $\gamma_3(m,T)$ when each $A^*_l$ is marginally stable, i.e. satisfies \eqref{eq:stable_mat_specrad}. In this case, we obtain via Proposition \ref{prop:stable_mat_trace_bound} (see also \eqref{eq:trace_bd_stable_temp}) the bounds 
\begin{align} \label{eq:gamma23_bds}
    \gamma_i(m,T) \lesssim T \sum_{l=1}^m d T^{\alpha_l} \kappa^2(P_l) \leq d m T^{\alpha + 1} \kappa^2_{\max}, \quad \text{for \rev{$i=2,3$}}.
\end{align}
For the term $\xi(m,T,\delta)$, we note that
\begin{equation} \label{eq:xi_bd}
    (1 \leq ) \ \log(1/\delta) \leq \xi(m,T,\delta) \lesssim \log\Big(\frac{d mT^{\alpha} \kappa^2_{\max}}{\delta} \Big)
\end{equation}
where the upper bound follows using \eqref{eq:gamma23_bds}. Then following the same steps as earlier, it is easy to see that for the choice 
    \begin{align} \label{eq:tau_tilde_exp_subs_cor_temp1}
        \tilde{\tau}^* = \min\set{\underbrace{\left(\frac{m T}{C_3' dR^4 \log\Big(\frac{dmT^{\alpha} \kappa^2_{\max}}{\delta}\Big)}\right)^{1/3}}_{= \bar{\tau}^*} , \ \max\set{\Big(\frac{2 m^2 S_m T}{R^2 d^2}\Big)^{1/3},1}},
    \end{align}
the condition on $m$ in \eqref{eq:m_con_subs_cor_stab} ensures \eqref{eq:m_cond_corr_subs_constr} and $\bar{\tau}^* \in [1,m]$. Thus $\tau^* = \lfloor \tilde{\tau}^* \rfloor \in [m]$ as well. 

The regime for $S_m$ in the statement follows from the condition $\bar{\tau}^* \leq \Big(\frac{2 m^2 S_m T}{R^2 d^2}\Big)^{1/3}$ with $\bar{\tau}^*$ as in \eqref{eq:tau_tilde_exp_subs_cor_temp1}, which implicitly implies $\Big(\frac{2 m^2 S_m T}{R^2 d^2}\Big)^{1/3} \geq 1$ (since $\bar{\tau}^* \geq 1$). Hence in this smoothness regime, $\tilde{\tau}^* = \bar{\tau}^*$. Plugging \rev{$\tau^* = \lfloor \tilde{\tau}^* \rfloor$} in \eqref{eq:lapl_smooth_path_graph_mainbd_1}, and using the bounds in \eqref{eq:gamma23_bds}, \eqref{eq:xi_bd} \rev{and \eqref{eq:gamma_1_bd_stable}}, we obtain the stated error bound after some minor simplifications.

%% file: exps.tex
\section{Simulations} \label{sec:sims}
We now provide some simulation results on a synthetic example to empirically validate the performance of the methods. The setup is as follows. We will consider $G$ to be the path graph on $m$ vertices. The matrices $A^*_1,\dots,A^*_m$ are generated as in Example \ref{ex:path_graph_holder} where 
\begin{equation*}
    f^*_{i,j} (x) = 4x^{\beta} - \sin\Big(\frac{2\pi i j}{d} x \Big)
\end{equation*}
for $\beta \in (0,1]$. Note that $f^*_{ij}$ is $\beta$-H\"older continuous and $S_m = O(m^{1-2\beta})$. In order to ensure that $\rho(A^*_l) \leq 1$ for each $l$ (which is what we assume in our analyses), we normalize $A^*_l$ by $\max\set{\rho(A^*_l)}_{l=1}^m$.

We fix $d = 10$ throughout, and choose $m \in \set{5,10,20,40,60,80,100,150}$. For a given value of $m$, we repeat the following steps $30$ times.
\begin{itemize}
    \item Generate trajectories for each $l \in [m]$ as in \eqref{eq:mult_linsys} with $\eta_{l,t} \stackrel{iid}{\sim} \calN(0,I_d)$.

    \item Obtain estimates for Laplacian smoothing and for the subspace constrained estimators. Here we set $\lambda = 20 * m^{4\beta/5}$ for the former, and $\tau = \min\set{[1.5 *  m^{1/3}], m}$ for the latter, with $[m]$ denoting the round operator (to the nearest integer). Note that the dependence on $m$ (and $S_m$) is consistent with the statement of Corollaries \ref{corr:lapl_smooth_strict_stab} and \ref{corr:subs_constr_strict_stab_path}.

    \item Upon obtaining an estimate $\est{A}$, we compute the root mean squared error (RMSE) which is $\frac{1}{\sqrt{m}} (\sum_{l=1}^m \norm{\est{A}_l - A^*_{l}}_F^2)^{1/2}$.
\end{itemize}
We average the RMSE over the $30$ runs for both methods, and plot it versus $m$. The results are shown in Figure \ref{fig:plots_rmse_versus_m_d_30_runs} for $T = 5$ (left panel) and $T = 10$ (right panel) for $\beta \in \set{1/4, 1/2, 1}$. As expected, the RMSE decreases as $m$ increases. The Laplacian smoothing estimator performs relatively better on the above example, although we did not make an effort to fine-tune the choice of $\lambda, \tau$. Interestingly, this estimator is seen to perform well for even a small value of $T$ which reinforces our earlier remark that the requirement on $T$ in Corollary \ref{corr:lapl_smooth_strict_stab} is an artefact of the analysis.
\begin{figure}[!htp]
\centering
\begin{subfigure}{.5\textwidth}
  \centering
  \includegraphics[width=0.9\linewidth]{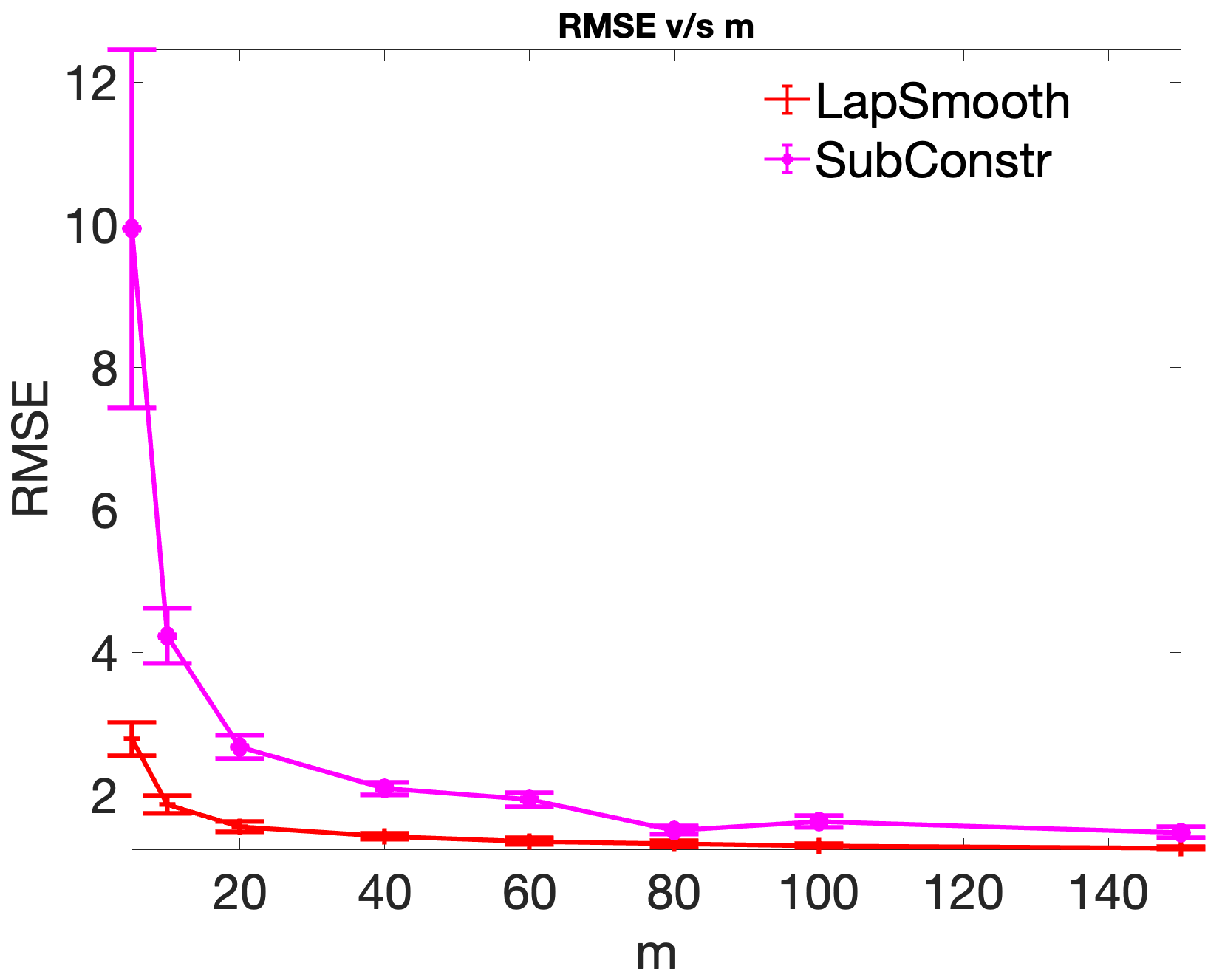}
  \caption{$S_m = O(\sqrt{m})$, $T = 5$}
\end{subfigure}%
\begin{subfigure}{.5\textwidth}
  \centering
  \includegraphics[width=0.9\linewidth]{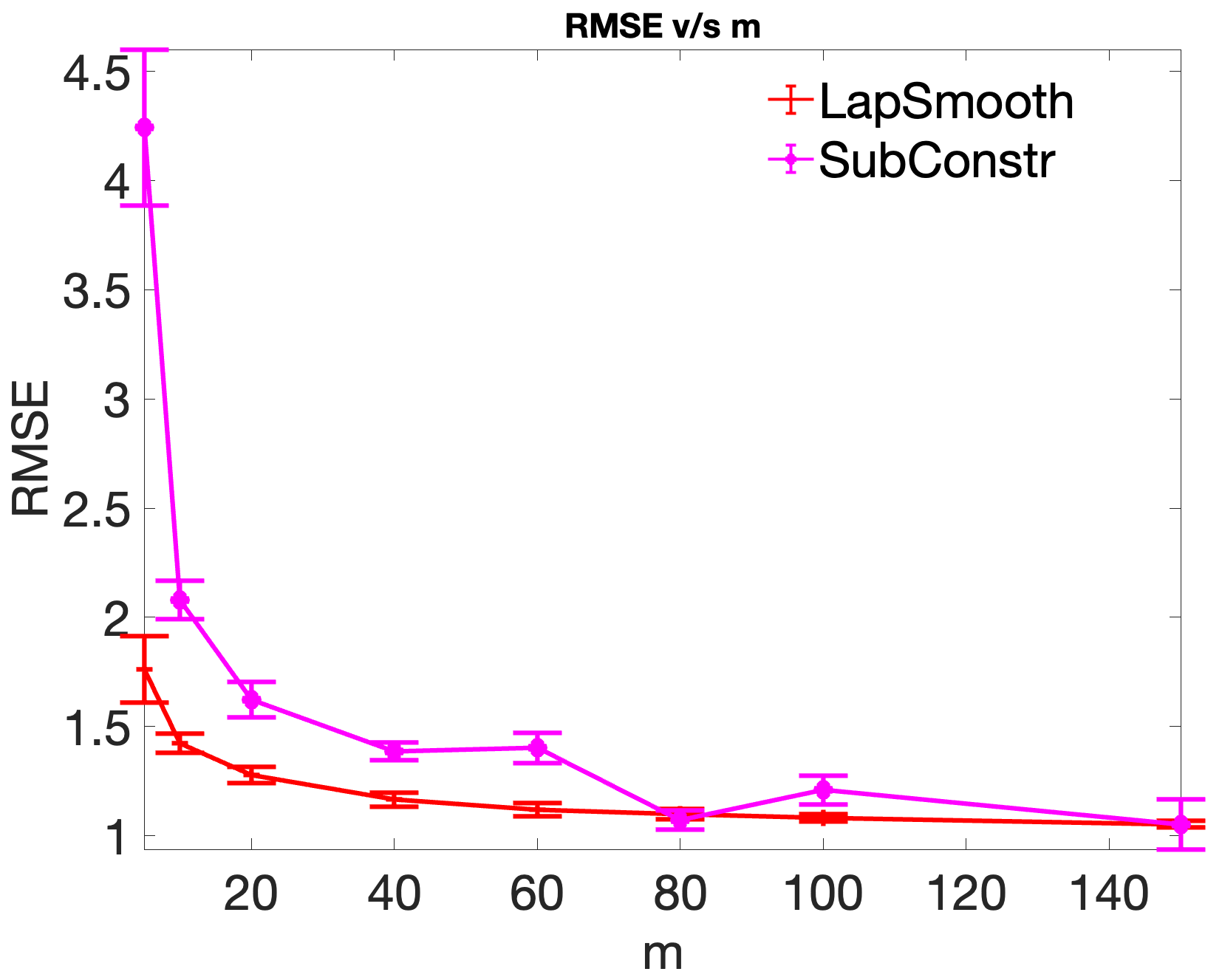}
  \caption{$S_m = O(\sqrt{m})$, $T = 10$}
\end{subfigure}%
\hfill
\begin{subfigure}{.5\textwidth}
  \centering
  \includegraphics[width=0.9\linewidth]{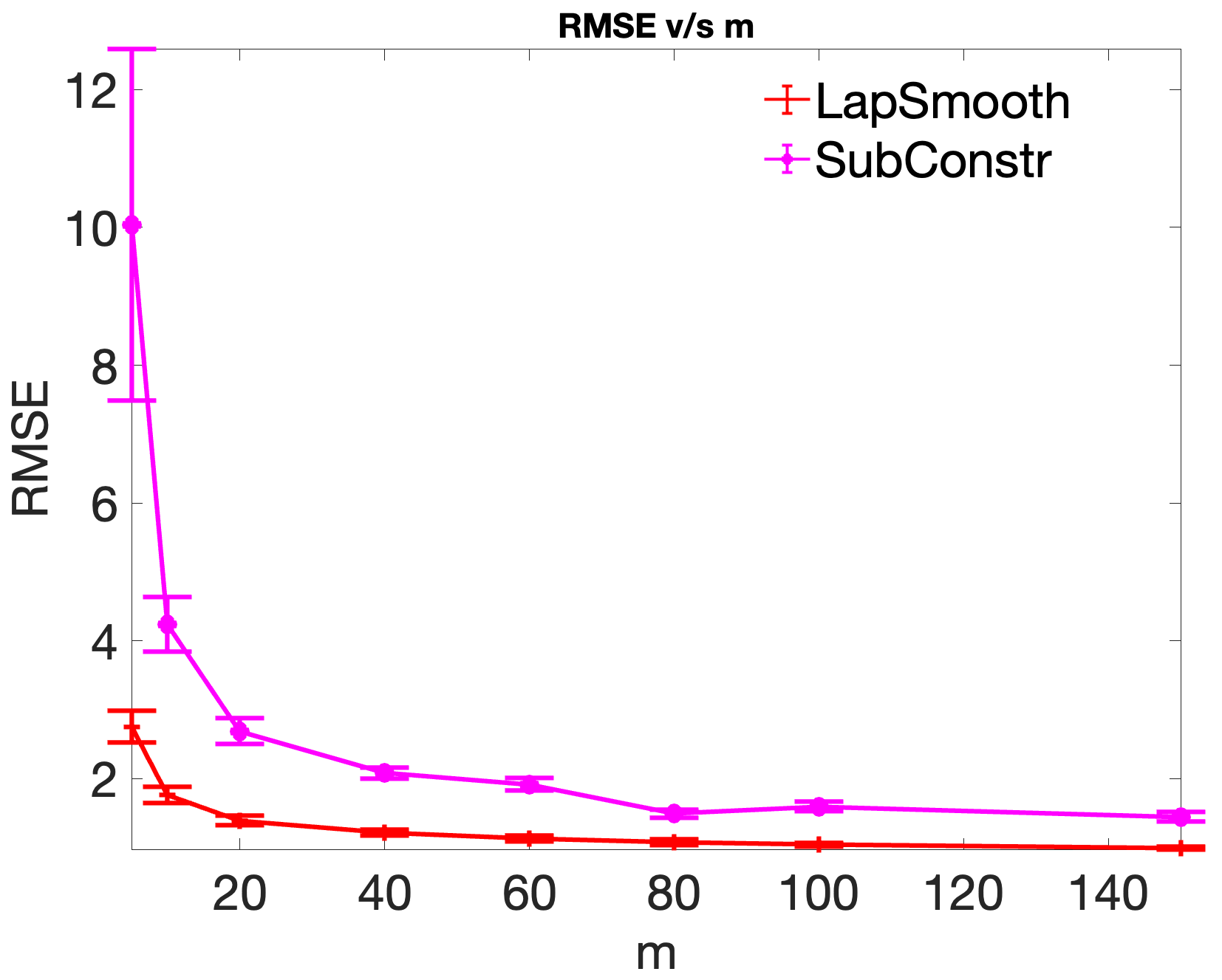}
  \caption{$S_m = O(1)$, $T = 5$}
\end{subfigure}%
\begin{subfigure}{.5\textwidth}
  \centering
  \includegraphics[width=0.9\linewidth]{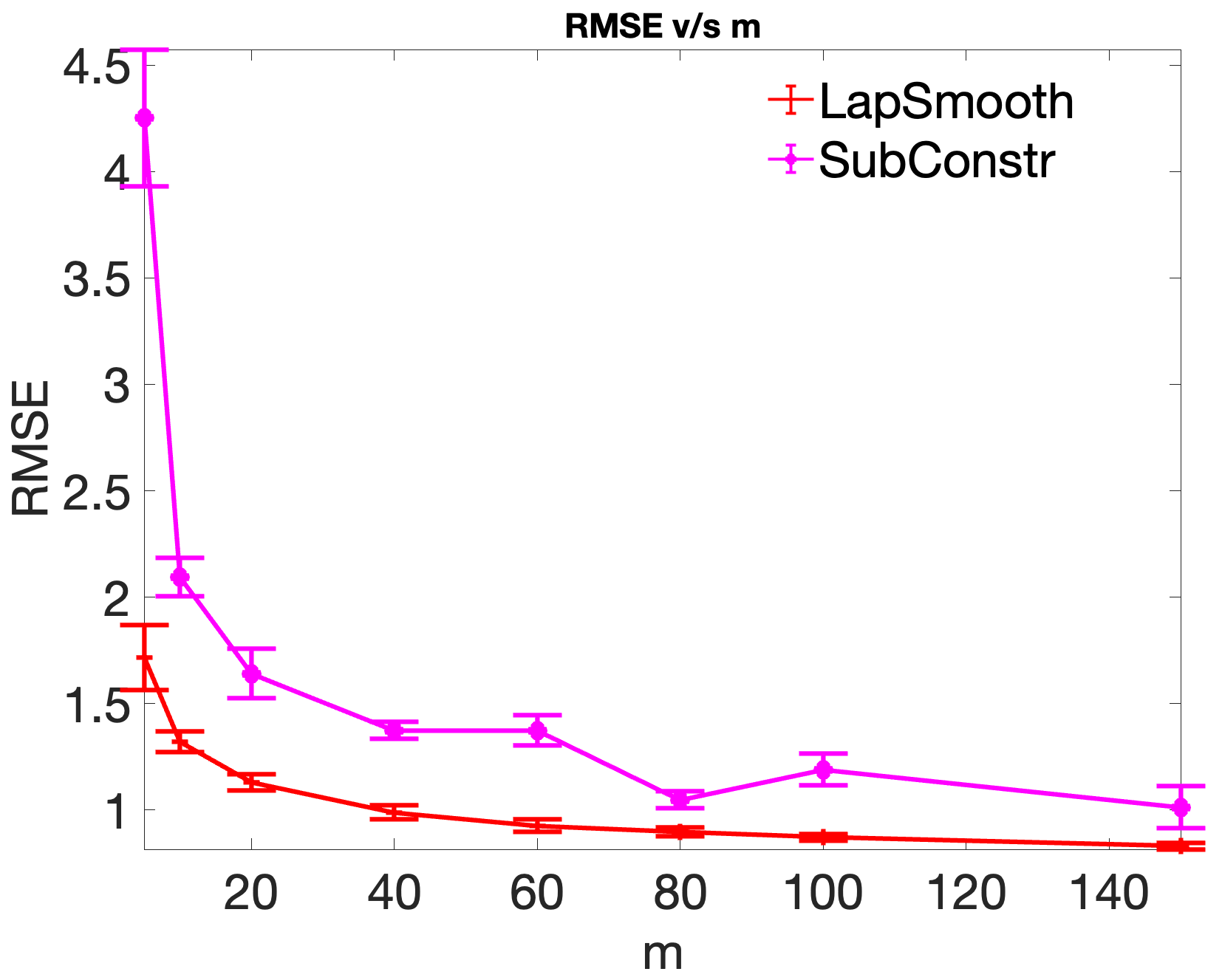}
  \caption{$S_m = O(1)$, $T = 10$}
\end{subfigure}
\hfill
\begin{subfigure}{.5\textwidth}
  \centering
  \includegraphics[width=0.9\linewidth]{figs/PLOTS_2/smooth_ID2_d10_T5_graph_path_numMC30_RMSE.png}
  \caption{$S_m = O(m^{-1})$, $T = 5$}
\end{subfigure}%
\begin{subfigure}{.5\textwidth}
  \centering
  \includegraphics[width=0.9\linewidth]{figs/PLOTS_2/smooth_ID2_d10_T10_graph_path_numMC30_RMSE.png}
  \caption{$S_m = O(m^{-1})$, $T = 10$}
\end{subfigure}
\caption{RMSE versus $m$ with $d=10$. Results averaged over $30$ trials with error bars depicting standard deviation.}
\label{fig:plots_rmse_versus_m_d_30_runs}
\end{figure}

%% file: concluding_remarks.tex
\section{Concluding remarks}
We considered the problem of joint learning of $m$ LDSs where each LDS --defined by a system matrix $A^*_l$ ($l \in [m]$) -- resides on the node of a graph $G = ([m], \calE)$. Assuming that the matrices $A^*_l$ are smooth w.r.t the graph $G$ -- as in \eqref{eq:quad_var_smooth} with parameter $S_m$ -- we formulate two estimators for jointly estimating $A^*_l$, along with non-asymptotic error bounds on the MSE. For the first estimator, namely Laplacian smoothing, our bounds hold provided $T$ scales logarithmically with $m$, and for a wide range of $S_m$. For the second estimator, our results hold without any requirement on $T$, albeit for a relatively narrow smoothness range, $S_m = o(m^{-1/3})$. For both these estimators, the MSE goes to zero as $m \rightarrow \infty$, at a rate typically polynomially fast w.r.t $m$. In our simulations, we find that both estimators perform well for small values of $T$.

An important direction for future work would be to improve the analyses for both the estimators. In particular, the ultimate goal would be to obtain error bounds on the MSE which hold for all smoothness regimes $S_m = o(m)$, and under no assumption on $T$. Another important direction would be to generalize the analysis to handle the situation where the $A^*_l$'s are not necessarily stable. We believe this can be done by leveraging ideas from the work of \cite{Sarkar19}, although the extension is certainly non-trivial.

%% file: appendix.tex
%
\input{appendix_laplacian_smoothing}

\input{appendix_subspace_constr}

\input{appendix_useful_tools}

%% file: appendix_laplacian_smoothing.tex
%
\section{Proofs from Section \ref{sec:laplacian_smooth_analysis}} \label{appsec:proofs_lapl_smooth}
\subsection{Proof of Proposition \ref{prop:bd_self_norm_laplsmooth}}
\rev{Denote the canonical basis of $\matR^m$ by $e_1,\dots,e_m$}. Using the definition of $Q$, we can write 
\begin{align*}
    Q^\top \eta &= \sum_{l=1}^m e_l \otimes \Big(\sum_{t=1}^T (x_{l,t} \otimes I_d) \eta_{l,t+1} \Big) = \sum_{l=1}^m \sum_{t=1}^T e_l \otimes (x_{l,t} \otimes \eta_{l,t+1}), \\
    Q^\top Q &= \sum_{l=1}^m \sum_{t=1}^T e_le_l^\top \otimes ((x_{l,t}x_{l,t}^\top) \otimes I_d) = \blkdiag(X_l X_l^\top \otimes I_d)_{l=1}^m.
\end{align*}
\rev{Denoting $\calF_t := \sigma(\eta_{l,1},\dots,\eta_{l,t})_{l=1}^m$, we obtain a filtration $(\calF_{t})_{t=1}^\infty$}. 

We now proceed to derive an analogous result to \cite[Lemma 9]{abbasi11} for our setting. 
\begin{lemma}
    For any $\mu \in \matR^{md^2}$, consider for any $t \geq 1$,
    \begin{align*}
        M_{t}^\mu = \exp\Bigg( \sum_{s=1}^t \sum_{l=1}^m\left[  \frac{\dotprod{\mu}{e_l \otimes (x_{l,s} \otimes \eta_{l,s+1})}}{R} - \frac{\norm{\mu}_{e_le_l^\top \otimes ((x_{l,s}x_{l,s}^\top) \otimes I_d)}^2}{2} \right] \Bigg).
    \end{align*}
    Let $\stoptime$ be a stopping time w.r.t $(\calF_{t})_{t=1}^\infty$. Then $M_{\stoptime}^\mu$ is well defined and $\expec[M_{\stoptime}^\mu] \leq 1$.
\end{lemma}
\begin{proof}
    We will show that $(M_{t}^\mu)_{t=1}^{\infty}$ is a supermartingale. To this end, denote 
    \begin{equation*}
        D_{t}^\mu := \exp\Bigg(  \sum_{l=1}^m\left[  \frac{\dotprod{\mu}{e_l \otimes (x_{l,t} \otimes \eta_{l,t+1})}}{R} - \frac{\norm{\mu}_{e_le_l^\top \otimes ((x_{l,t}x_{l,t}^\top) \otimes I_d)}^2}{2} \right] \Bigg). 
    \end{equation*}
    Observe that $D_{t}^\mu, M_{t}^\mu$ are both $\calF_{t+1}$-measurable. Hence, we obtain
    \begin{align*}
        \expec[D_{t}^\mu \vert \calF_{t}] 
        = \prod_{l=1}^m \expec\left[\exp\Bigg(    \frac{\dotprod{\mu}{e_l \otimes (x_{l,t} \otimes \eta_{l,t+1})}}{R} - \frac{\norm{\mu}_{e_le_l^\top \otimes ((x_{l,t}x_{l,t}^\top) \otimes I_d)}^2}{2} \Bigg) \bigg| \calF_t \right]  
        \leq 1 \ \text{a.s.}
    \end{align*}
    using independence over $l$, and since $\eta_{l,t}$ has $R$-subgaussian entries (for each $l,t$) by assumption. This in turn implies
    \begin{align*}
        \expec[M_{t}^\mu \vert \calF_{t}]  \leq D_1^\mu  D_2^\mu \cdots  D_{t-1}^\mu = M_{t-1}^{\mu}
    \end{align*}
    meaning that $(M_{t}^\mu)_{t=1}^{\infty}$ is a supermartingale, and also $\expec[M_{t}^\mu] \leq 1$. The arguments for showing that $M_{\stoptime}^\mu$ is well-defined, and $\expec[M_{\stoptime}^\mu] \leq 1$ are identical to \cite[Lemma 9]{abbasi11}, and hence omitted.
\end{proof}
From here onwards, we obtain in a completely analogous manner as in the proof of \cite[Theorem 1]{abbasi11}, that 
\begin{equation} \label{eq:Q_bd_temp_1}
    \prob\left(\underbrace{\norm{Q^\top \eta}^2_{{(Q^\top Q + \Mbar )}^{-1}} \leq 2 R^2 \log \left(\frac{\det((Q^\top Q + \Mbar )^{1/2}) \det(\Mbar ^{-1/2})}{\delta} \right)}_{=: \calE_{2a}} \right) \leq \delta
\end{equation}
for any given $\Mbar \succ 0$.

It remains to derive an upper  L{\"o}wner bound on $Q^\top Q$, and to choose $\Mbar$ appropriately. 
\begin{lemma} \label{lem:Qmat_bd}
For any $\delta \in (0,1)$, consider the event
\begin{equation} \label{eq:qq_lowner_bd}
    \calE_{2b} := \set{Q^\top Q \preceq \gamma_1(T,m,\delta) (I_m \otimes I_{d^2})}
\end{equation}
   where $\gamma_1(T,m,\delta)$ is as in \eqref{eq:gamma_def_laplsmooth_strict_stab}. Then there exists a constant $c > 0$ such that for any $\delta \in (0,e^{-c})$, $\prob(\calE_{2b}) \geq 1-\delta$.
\end{lemma}
The proof is detailed in Section \ref{appsec:proof_lemma_Qmat_bd}. Now choosing $\Mbar  = \lambda(L \otimes I_{d^2}) + \frac{T}{4} (I_m \otimes I_{d^2})$, we have $Q^\top Q + \Mbar = \Mlamtil$. Furthermore, on the event in Lemma \ref{lem:Qmat_bd}, we obtain the bound 
\begin{align*}
    Q^\top Q + \Mbar  
    &\preceq \lambda(L \otimes I_{d^2}) + \Big(\gamma_1(T,m,\delta) + \frac{T}{4} \Big) (I_m \otimes I_{d^2}) \\
    &\preceq  \Big(\gamma_1(T,m,\delta) + \frac{T}{4} \Big)[(v_m v_m^\top) \otimes I_{d^2}] + \sum_{l=1}^{m-1} \left(\lambda \lambda_l + \gamma_1(T,m,\delta ) + \frac{T}{4} \right) [(v_l v_l^\top) \otimes I_{d^2}]
\end{align*}
which in turn implies 
\begin{align*}
    \det((Q^\top Q + \Mbar )^{1/2}) &\leq \left[\prod_{l=1}^{m-1} \Big(\lambda \lambda_l + \gamma_1(T,m,\delta ) + \frac{T}{4} \Big)^{d^2/2} \right] \Big(\gamma_1(T,m,\delta ) + \frac{T}{4} \Big)^{d^2/2}, \\
    \text{ and } \quad \det(\Mbar^{1/2}) &= (T/4)^{d^2/2} \Big(\prod_{l=1}^{m-1} (\lambda \lambda_l + \frac{T}{4})^{d^2/2} \Big).
\end{align*}
Plugging the above bounds in \eqref{eq:Q_bd_temp_1} then leads to the statement of Proposition \ref{prop:bd_self_norm_laplsmooth} via a union bound, as $\calE_{2a} \cap \calE_{2b} \subseteq \calE_{2}$.

%
\subsection{Proof of Lemma \ref{lem:Qmat_bd}} \label{appsec:proof_lemma_Qmat_bd}
Denoting $Y_{l,T} :=  \sum_{t=1}^T x_{l,t} x_{l,t}^\top$, we can write $\norm{Q^\top Q}_2 = \max_l \norm{Y_{l,T}}_2$. Now we bound $\norm{Y_{l,T}}_2$ as 
\begin{equation*}
    \norm{Y_{l,T}}_2 \leq \sum_{t=1}^T \norm{x_{l,t} x_{l,t}^\top}_2 = \sum_{t=1}^T x_{l,t}^\top x_{l,t}. 
\end{equation*}
Since for any $l \in [m]$, 
\begin{equation*}
\begin{bmatrix}
   x_{l,1} \\ 
   x_{l,2} \\ 
  \vdots \\
  x_{l,T}
 \end{bmatrix} =
\underbrace{\begin{bmatrix}
   I_n & 0 & \hdots & 0\\ 
   A_l^* & I_n & \hdots & 0\\ 
  \vdots  &  & \ddots & \vdots\\
  (A_l^*)^{T-1} & \hdots & A_l^* & I_n
 \end{bmatrix}}_{\tilde{A}_l}
\underbrace{\begin{bmatrix}
   \eta_{l,1} \\ 
   \eta_{l,2} \\ 
  \vdots \\
  \eta_{l,T}
 \end{bmatrix}}_{\tilde{\eta}_l},
\end{equation*}
we obtain $\norm{Y_{l,T}}_2 \leq \sum_{l=1}^m \tilde{\eta}_l^\top \tilde{A}_l^\top \tilde{A}_l \tilde{\eta}_l$. Since each entry of $\tilde{\eta}_l$ is $R$-subgaussian, this implies $\norm{\tilde{\eta}_{l,i}}_{\psi_2} \leq C R$ (for each $l,i$) for a universal constant $C>0$. Also, since $\expec[\tilde{\eta}_l^\top \tilde{A}_l^\top \tilde{A}_l \tilde{\eta}_l] = \Tr(\tilde{A}_l^\top \tilde{A}_l)$, we obtain via Hanson-Wright inequality \cite{Hanson71, rudelson2013} that 
\begin{equation} \label{eq:Q_HW_temp_1}
    \prob\Big(\abs{\tilde{\eta}_l^\top \tilde{A}_l^\top \tilde{A}_l \tilde{\eta}_l - \Tr(\tilde{A}_l^\top \tilde{A}_l)} \geq t \Big) \leq 2\exp\Bigg(-c \min\set{\frac{t}{C^2 R^2 \norm{\tilde{A}_l^\top \tilde{A}_l}_2}, \frac{t^2}{C^4 R^4 \norm{\tilde{A}_l^\top \tilde{A}_l}_F^2}} \Bigg)
\end{equation}
for some constant $c > 0$. Choosing $t = \frac{\norm{\tilde{A}_l^\top \tilde{A}_l}_F^2}{\norm{\tilde{A}_l^\top \tilde{A}_l}_2} (\frac{C^2 R^2 \log(1/\delta)}{c})$, the RHS of \eqref{eq:Q_HW_temp_1} is bounded by
\begin{align*}
2\exp\Bigg(-\log(1/\delta) \underbrace{\frac{\norm{\tilde{A}_l^\top \tilde{A}_l}_F^2}{\norm{\tilde{A}_l^\top \tilde{A}_l}_2^2}}_{\geq 1} \min \set{1, \frac{\log(1/\delta)}{c}} \Bigg)    
\end{align*}
which in turn is bounded by $2\delta$ if $\delta \leq e^{-c}$. Since $\Tr(\tilde{A}_l^\top \tilde{A}_l) = \sum_{t=0}^{T-1} \Tr(\Gamma_t(A_l^*))$ and $\norm{\tilde{A}_l^\top \tilde{A}_l}_F^2 \leq \norm{\tilde{A}_l^\top \tilde{A}_l}_2 \Tr(\tilde{A}_l^\top \tilde{A}_l)$, the statement of the lemma now follows via a union bound over $[m]$, and replacing $\delta$ by $\delta/m$.

%% file: appendix_subspace_constr.tex
%
\section{Proofs from Section \ref{sec:subspace_LS_analysis}} \label{appsec:sub_constr_proofs}
%
\subsection{Proof of Lemma \ref{lem:eigval_control_mat_subspace_method}}
Recall $Y_{l,T} := \sum_{t=1}^{T} x_{l,t} x_{l,t}^\top$. We begin by expanding $Q^\top Q$ as $Q^\top Q = P  + (Z + Z^\top) + N$,  where
\begin{align*}
    P & = \diag(P_l \otimes I_d), \quad P_l = A_l Y_{l,T-1} A_l^\top; \\
    Z & = \diag(Z_l \otimes I_d), \quad Z_l = \sum_{t=1}^T A_l x_{l,t-1} \eta_{l,t}^\top;  \\
  \text{ and }  N &  = \diag(N_l \otimes I_d), \quad N_l = \sum_{t=1}^T \eta_{l,t} \eta_{l,t}^\top.
\end{align*}
Let us further define the quantities 
\begin{align} \label{eq:f_l_def}
    f_l := \begin{bmatrix}
   v_{m,l} \\ 
  \vdots \\
  v_{m-\tau+1,l}
 \end{bmatrix} \quad \text{ and } \quad F_l = f_l f_l^\top.
\end{align}
Note that since $G$ is $\theta$-delocalized, we have $\norm{f_l}_2^2 \leq \theta\tau/m$ for each $l$.  
Then from the definition of $\Vtau$ in \eqref{eq:Vtau_def}, we can write $\Vtau^\top Q^\top Q \Vtau$ as
\begin{align*}
    \Vtau^\top Q^\top Q \Vtau 
    &= \Vtau^\top P \Vtau + \Vtau^\top (Z + Z^\top) \Vtau + \Vtau^\top N \Vtau
 \\    
    &=\sum_{l=1}^m (F_l \otimes P_l) \otimes I_d + \sum_{l=1}^m (F_l \otimes Z_l) \otimes I_d +  \sum_{l=1}^m (F_l \otimes Z_l^\top) \otimes I_d + \sum_{l=1}^m  (F_l \otimes N_l) \otimes I_d
\end{align*}
Now we will proceed to derive a lower bound on $u^\top (\Vtau^\top Q^\top Q \Vtau) u$ holding for all $u \in \mathbb{S}^{\tau d^2 - 1}$. 

To this end, we first invoke \cite[Proposition 7.1]{Sarkar19}, which in our notation states the following. 
\begin{proposition}[{\cite[Proposition 7.1]{Sarkar19}}] \label{prop:useful_ineq_prop}
For $M \succ 0$, suppose $$\norm{(\Vtau^\top P \Vtau + M)^{-1/2} (\Vtau^\top Z \Vtau)}_2 \leq \zeta.$$ Then for any $u \in \mathbb{S}^{\tau d^2 - 1}$ such that
\begin{align*}
    u^\top (\Vtau^\top P \Vtau) u = \alpha \quad \text{ and } \quad u^\top M u = \alpha',
\end{align*}
we have $\norm{\Vtau^\top Z^\top \Vtau u}_2 \leq \zeta(\alpha + \alpha')^{1/2}$.
\end{proposition}
Choosing $M = \Mtil \otimes I_d $ for $\Mtil \succ 0$, it holds
\begin{align}
    \norm{(\Vtau^\top P \Vtau + M)^{-1/2} (\Vtau^\top Z \Vtau)}_2 &= \Big\|(\sum_{l=1}^m (F_l \otimes P_l) + \Mtil)^{-1/2} (\sum_{l=1}^m F_l \otimes Z_l) \Big\|_2 \nonumber \\
    &= \sup_{w \in \mathbb{S}^{\tau d - 1}} \Big\|(\sum_{l=1}^m (F_l \otimes P_l) + \Mtil)^{-1/2} (\sum_{l=1}^m F_l \otimes Z_l) w \Big\|_2  \label{eq:temp_bd_1_subsp_proof}
\end{align}
and so we focus now on bounding the RHS of \eqref{eq:temp_bd_1_subsp_proof}. 
\begin{proposition} \label{prop:spec_norm_bd_subsp_constr}
  Let $\Mtil = T I_{\tau d}$. Then there exists a constant $C > 0$ such that for any $\delta \in (0,1)$, it holds for the event
    \begin{equation*}
        \calE_{3a} := \set{\Big\|(\sum_{l=1}^m (F_l \otimes P_l) + \Mtil)^{-1/2} (\sum_{l=1}^m F_l \otimes Z_l) \Big\|_2 \leq C R\tau \sqrt{\frac{\theta d}{m}} \Bigg(\log\Big(\frac{1}{\delta} \Big) + \log\Big(\frac{\theta \tau \gamma_2(m,T)}{mT} + 1 \Big) \Bigg)^{1/2}}
    \end{equation*}
    that $\prob(\calE_{3a}) \geq 1-2\delta$.
\end{proposition}
The arguments in the proof are analogous to that in the proof of Proposition \ref{prop:bd_self_norm_laplsmooth}. 

The next step is to obtain a lower (L\"owner) bound on $\sum_{l=1}^m  (F_l \otimes N_l) \otimes I_d$, this is outlined in the proposition below. The proof is along the lines of \cite[Theorem 4.6.1]{HDPbook}, but does not, unfortunately, follow directly from this theorem since the random vector $f_l \otimes \eta_{l,t}$ is not isotropic.
\begin{proposition} \label{prop:conc_excit_subsp_constr}
   Denote the event $$\calE_{3b} := \set{\sum_{l=1}^m  (F_l \otimes N_l) \otimes I_d \succeq \frac{3T}{4} I_{\tau d^2}}.$$ There exist constants $c_1, c_2 > 0$ such that if
    \begin{equation*}
        c_1 R^2\max \set{\sqrt{\frac{\theta^3 \tau^3 d T}{m}}, \frac{\theta^2 \tau^2 d}{m}} \leq \frac{T}{4},
    \end{equation*}
    then $\prob(\calE_{3b}) \geq 1-2\exp(-c_2 \theta\tau d)$.
\end{proposition}
The proof is now completed using Propositions \ref{prop:useful_ineq_prop},\ref{prop:spec_norm_bd_subsp_constr} and \ref{prop:conc_excit_subsp_constr} as follows. For any $u \in \mathbb{S}^{\tau d^2 - 1}$, denote $\kappa^2 = u^\top (\Vtau^\top P \Vtau) u$, and choose $\Mtil = T I_{\tau d}$ so that $u^\top M u = u^\top (\Mtil \otimes I_d) u = T$. Condition on the event $\calE_{3a} \cap \calE_{3b}$. Then Proposition \ref{prop:useful_ineq_prop} implies
\begin{align*}
    \norm{\Vtau^\top Z^\top \Vtau u}_2 \leq (\kappa^2 + T)^{1/2} C R\tau \sqrt{\frac{\theta d}{m}} \Bigg(\log\Big(\frac{1}{\delta} \Big) + \log\Big(\frac{\theta \tau \gamma_2(m,T)}{mT} + 1 \Big) \Bigg)^{1/2}, \quad \forall u \in \mathbb{S}^{\tau d^2 - 1},
\end{align*}
with the same bound holding for $\norm{\Vtau^\top Z \Vtau u}_2$ (uniformly over $u \in \mathbb{S}^{\tau d^2 - 1}$). Further invoking Proposition \ref{prop:conc_excit_subsp_constr}, this implies for all $u \in \mathbb{S}^{\tau d^2 - 1}$ that
\begin{align*}
    u^\top (\Vtau Q^\top Q \Vtau) u 
    &\geq \kappa^2 - 2C R\tau(\kappa^2 + T)^{1/2}  \sqrt{\frac{\theta d}{m}} \Bigg(\log\Big(\frac{1}{\delta} \Big) + \log\Big(\frac{\theta \tau \gamma_2(m,T)}{mT} + 1 \Big) \Bigg)^{1/2} + \frac{3T}{4} \\
    &\geq \kappa^2 - (\kappa^2 + T)^{1/2}\frac{\sqrt{T}}{2} + \frac{3T}{4} \tag{if $T$ satisfies \eqref{eq:Tcond_subs_constr_theorem} for suitable $C > 0$} \\
    &\geq \kappa^2 - (1+\frac{\kappa^2}{2T})\frac{T}{2} + \frac{3T}{4} \\
    &\geq \frac{3 \kappa^2}{4} + \frac{T}{4} \geq \frac{T}{4},
\end{align*}
which completes the proof.

%
\subsection{Proof of Proposition \ref{prop:spec_norm_bd_subsp_constr}}
Denote for convenience $G_{l,t} := A_l x_{l,t-1}$ and consider $W = [w_1 \ w_2 \ \cdots w_{\tau}] \in \matR^{d \times \tau}$ with $\norm{W}_F = 1$. Let $\Mtil \succ 0$ be arbitrary for now. Then for $w = \vect(W) \in \mathbb{S}^{\tau d - 1}$, one can verify with a little effort that the RHS of \eqref{eq:temp_bd_1_subsp_proof} (for a given $w$) equals
\begin{align} \label{eq:temp_bd_2_subsp_proof}
    \Big\|(\sum_{l=1}^m (F_l \otimes P_l) &+ \Mtil)^{-1/2} (\sum_{l=1}^m F_l \otimes Z_l) w \Big\|_2 \nonumber \\
    &= \left\|\Big(\sum_{t=1}^T\sum_{l=1}^m (f_l \otimes G_{l,t})(f_l \otimes G_{l,t})^\top + \Mtil \Big)^{-1/2} \Big(\sum_{t=1}^T\sum_{l=1}^m \dotprod{W f_l}{\eta_{l,t}} (f_l \otimes G_{l,t}) \Big)\right\|_2.
\end{align}
The RHS of \eqref{eq:temp_bd_2_subsp_proof} can be bounded in an analogous manner to Proposition \ref{prop:bd_self_norm_laplsmooth}, which along with a standard covering argument gives a bound on the RHS of \eqref{eq:temp_bd_1_subsp_proof}.

To this end, with $(\calF_t)_{t \geq 1}$ a filtration as before, we have the following analogous result to \cite[Lemma 9]{abbasi11} for our setting.
\begin{lemma}
 For any $\mu \in \matR^{\tau d}$, consider for  $t \geq 1$,
    \begin{align*}
        M_{t}^\mu := \exp\Bigg( \sum_{s=1}^t \sum_{l=1}^m \left[\frac{\dotprod{\mu}{f_l \otimes G_{l,s}} \dotprod{W f_l}{\eta_{l,s}}}{R \sqrt{\theta \tau/m}} - \frac{\dotprod{\mu}{f_l \otimes G_{l,s}}^2}{2} \right] \Bigg).
    \end{align*}
    Let $\stoptime$ be a stopping time w.r.t $(\calF_{t})_{t=1}^\infty$. Then $M_{\stoptime}^\mu$ is well defined and $\expec[M_{\stoptime}^\mu] \leq 1$.   
\end{lemma}
\begin{proof}
We will show that $(M_{t}^\mu)_{t=1}^{\infty}$ is a supermartingale. To this end, denote 
 \begin{equation*}
        D_{t}^\mu := \exp\Bigg(\sum_{l=1}^m \left[\frac{\dotprod{\mu}{f_l \otimes G_{l,t}} \dotprod{W f_l}{\eta_{l,t}}}{R\sqrt{\theta \tau/m}} - \frac{\dotprod{\mu}{f_l \otimes G_{l,t}}^2}{2} \right] \Bigg). 
 \end{equation*}
  Observe that $D_{t}^\mu, M_{t}^\mu$ are both $\calF_{t}$-measurable. Hence, we obtain
   \begin{align*}
        \expec[D_{t}^\mu \vert \calF_{t-1}] 
        = \prod_{l=1}^m \expec\left[\exp\Bigg(\frac{\dotprod{\mu}{f_l \otimes G_{l,t}} \dotprod{W f_l}{\eta_{l,t}}}{R\sqrt{\theta\tau/m}} - \frac{\dotprod{\mu}{f_l \otimes G_{l,t}}^2}{2} \Bigg) \bigg\vert \calF_{t-1}\right]
        \leq 1 \ \text{a.s.}
  \end{align*}
  using independence over $l$, and the fact
  \begin{align*}
  &\expec\Big[\exp \Big(\dotprod{\mu}{f_l \otimes G_{l,t}} \dotprod{W f_l}{\eta_{l,t}} \Big) \big\vert \calF_{t-1} \Big]   \\
  &\leq  \exp\Bigg(\frac{\dotprod{\mu}{f_l \otimes G_{l,s}}^2\norm{W f_l}_2^2}{2} \Bigg) 
  \leq \exp\Bigg(\frac{\dotprod{\mu}{f_l \otimes G_{l,s}}^2(R^2\theta \tau / m)}{2} \Bigg).  
  \end{align*}  
  The last inequality follows from $\norm{f_l}_2^2 \leq \theta\tau/m$. This in turn implies $\expec[M_{t}^\mu \vert \calF_{t-1}] \leq M_{t-1}^{\mu}$
  meaning that $(M_{t}^\mu)_{t=1}^{\infty}$ is a supermartingale, and also $\expec[M_{t}^\mu] \leq 1$. The arguments for showing that $M_{\stoptime}^\mu$ is well-defined, and $\expec[M_{\stoptime}^\mu] \leq 1$ is identical to \cite[Lemma 9]{abbasi11}, and hence omitted.
\end{proof}
Let us now denote for convenience
\begin{align*}
    U_T &= \sum_{l=1}^m (F_l \otimes P_l) = \sum_{t=1}^T \sum_{l=1}^m (f_l \otimes G_{l,t})(f_l \otimes G_{l,t})^\top, \\ 
    \text{ and } \ \Ubar_T &= U_T + \Mtil.
\end{align*}
Then, we obtain in a completely analogous manner as in the proof of \cite[Theorem 1]{abbasi11}, that for any $\delta\in (0,1)$, 
\begin{equation} \label{eq:subsp_temp_4}
\prob\Bigg( \Big\|\sum_{t=1}^T\sum_{l=1}^m \dotprod{W f_l}{\eta_{l,t}} (f_l \otimes G_{l,t}) \Big\|^2_{\Ubar_T^{-1}} \geq \frac{2R^2 \theta \tau}{m} \log \Big[\frac{\det(\Ubar_T)^{1/2} \det(\Mtil)^{-1/2}}{\delta} \Big]\Bigg)  \leq \delta.  
\end{equation}
Recalling the identity in \eqref{eq:temp_bd_2_subsp_proof} we then obtain via a standard covering number argument (see e.g., \cite[Lemma 4.4.1]{HDPbook}), for any $z > 0$, 
\begin{align*} 
&\prob\Bigg( \Big\|\Big(\sum_{l=1}^m (F_l \otimes P_l) + \Mtil \Big)^{-1/2} \Big(\sum_{l=1}^m F_l \otimes Z_l \Big)  \Big\|_2^2 \geq z^2 \Bigg)  \\
&\leq 5^{\tau d} \sup_{u \in \mathbb{S}^{\tau d - 1}} \prob\Bigg(\Big\| \Big(\sum_{l=1}^m (F_l \otimes P_l) + \Mtil \Big)^{-1/2} \Big(\sum_{l=1}^m F_l \otimes Z_l \Big) w \Big\|_2^2 \geq z^2/4 \Bigg)  
\end{align*}
Using this with \eqref{eq:subsp_temp_4}, we then obtain
\begin{align} \label{eq:subsp_temp_5}
\prob\Bigg(\Big\|\Big(\sum_{l=1}^m (F_l \otimes P_l) + \Mtil \Big)^{-1/2} \Big(\sum_{l=1}^m F_l \otimes Z_l \Big)  \Big\|_2^2 \geq \frac{8 R^2 \theta \tau}{m} \log \Big[\frac{5^{\tau d}\det(\Ubar_T)^{1/2} \det(\Mtil)^{-1/2}}{\delta} \Big] \Bigg) \leq \delta.
\end{align}
It remains to upper (L\"owner) bound $U_T$.
\begin{lemma} \label{lem:lowner_bd_Umat}
    For any $\delta \in (0,1)$, $$\prob\left(U_T \preceq \frac{\theta \tau \gamma_2(m,T)}{\delta m}  I_{\tau d} \right) \geq 1-\delta. $$
\end{lemma}
\begin{proof}
    Note that $\sum_{t=1}^T G_{l,t} G_{l,t}^\top = A^*_l Y_{l,T-1} (A^*_l)^\top$ where we recall $Y_{l,T} :=  \sum_{t=1}^T x_{l,t} x_{l,t}^\top$. Hence $\norm{U_T}_2$ can be bounded as
    \begin{align*}
        \norm{U_T}_2 &= \Big\|\sum_{l=1}^m \Big((f_l f_l^\top) \otimes (A^*_l Y_{l,T-1} (A^*_l)^\top )\Big) \Big\|_2 \\ 
        &\leq \sum_{l=1}^m \norm{f_l}_2^2 \ \norm{A^*_l Y_{l,T-1} (A^*_l)^\top}_2 \\
        &\leq (\theta\tau/m)  \sum_{l=1}^m \norm{A^*_l Y_{l,T-1} (A^*_l)^\top}_2 \tag{since $\norm{f_l}_2^2 \leq \theta\tau/m$}\\
        &\leq (\theta\tau/m) \sum_{l=1}^m \sum_{t=1}^{T-1} x_{l,t}^\top  (A^*_l)^\top A^*_l x_{l,t} \\
        &= (\theta\tau/m) \sum_{l=1}^m \tilde{\eta}_l^\top \tilde{A}_l^\top (I_T \otimes (A^*_l)^\top) (I_T \otimes A^*_l) \tilde{A}_l \tilde{\eta}_l,
    \end{align*}
    where the last equality uses notation from Section \ref{appsec:proof_lemma_Qmat_bd}. This implies 
\begin{align*}
    \expec\Big[\norm{U_T}_2 \Big] 
    &\leq (\theta\tau/m) \sum_{l=1}^m \expec\Big[\tilde{\eta}_l^\top \tilde{A}_l^\top (I_T \otimes (A^*_l)^\top) (I_T \otimes A^*_l) \tilde{A}_l\tilde{\eta}_l \Big] \\
    &= (\theta\tau/m) \sum_{l=1}^m \Tr\Big(\tilde{\eta}_l^\top \tilde{A}_l^\top (I_T \otimes (A^*_l)^\top) (I_T \otimes A^*_l) \tilde{A}_l\tilde{\eta}_l \Big) \\
    &= (\theta\tau/m) \sum_{l=1}^m \sum_{t=1}^T \Tr(\Gamma_t(A^*_l) - I_d) = (\theta\tau/m)  \gamma_2(m,T).
 \end{align*}
 Applying Markov's inequality then completes the proof.
\end{proof}
Choosing $\Mtil = T I_{\tau d}$, Lemma \ref{lem:lowner_bd_Umat} implies
$\Ubar_T \preceq \Big(1 + \frac{2\tau \gamma_2(m,T)}{\delta m} \Big)  I_{\tau d}$ which upon plugging in \eqref{eq:subsp_temp_5}, and some minor simplification, leads to the statement of the proposition. 

%
\subsection{Proof of Proposition \ref{prop:conc_excit_subsp_constr}}
Since Kronecker product with identity matrix preserves the semi-definite order, we will focus on bounding $\sum_{l=1}^m  (F_l \otimes N_l) = \sum_{l=1}^m \sum_{t=1}^T (f_l \otimes \eta_{l,t}) (f_l \otimes \eta_{l,t})^\top$. The idea is to follow the steps outlined in \cite[Theorem 4.6.1]{HDPbook} and adapt it to our setting. 

While the vectors $(f_l \otimes \eta_{l,t})_{l,t}$ are independent, they are not isotropic, and this prevents us from using \cite[Theorem 4.6.1]{HDPbook} directly. Indeed, for any $U = [u_1 \ u_2 \ \cdots u_{\tau}]\in \matR^{d \times \tau}$ with $\norm{U}_F = 1$ and denoting $u=\vect(U)$, note that 
\begin{align*}
    \expec\Big[\dotprod{u}{f_l \otimes \eta_{l,t}}^2 \Big] = \expec \Big[\dotprod{\eta_{l,t}}{\sum_{i=1}^\tau f_{l,i} u_i}^2 \Big] = \norm{U f_l}_2^2
\end{align*}
which is not necessarily equal to $\norm{U}_F^2$ in general, and could also be zero depending on $U$. Nevertheless the idea in the proof of \cite[Theorem 4.6.1]{HDPbook} still applies -- we will first show a high probability bound on $u^\top(\sum_{l=1}^m  (F_l \otimes N_l))u$ for a given $u \in \mathbb{S}^{\tau d - 1}$, and then use a standard covering number argument to obtain a bound holding uniformly over $\mathbb{S}^{\tau d - 1}$.

To this end, first note that 
\begin{equation*}
    \expec\Big[\sum_{l=1}^m  (F_l \otimes N_l)\Big] = \sum_{l=1}^m \sum_{t=1}^T (f_l f_l^\top) \otimes \underbrace{\expec[\eta_{l,t} \eta_{l,t}^\top]}_{= I_d} =  \sum_{t=1}^T \Big(\underbrace{\sum_{l=1}^m f_l f_l^\top}_{=I_\tau} \Big) \otimes I_d = T I_{\tau d}.
\end{equation*}
Let $\calN_{\epsilon}$ be an $\epsilon$-net for $\mathbb{S}^{\tau d - 1}$, then we know from \cite[Ex. 4.4.3]{HDPbook} that 
\begin{align*}
    \Big\|\frac{1}{T} \sum_{l=1}^m  (F_l \otimes N_l) - I_{\tau d}  \Big\|_2 \leq \frac{1}{1-2\epsilon} \sup_{u \in \calN_{\epsilon}} \abs{\dotprod{\Big(\frac{1}{T}\sum_{l=1}^m  (F_l \otimes N_l) - I_{\tau d} \Big)u}{u}}.
\end{align*}
Moreover, we also know that $\abs{\calN_{\epsilon}} \leq (1+(2/\epsilon))^{\tau d}$. Choose $\epsilon = 1/4$ from now onwards, so, $\abs{\calN_{1/4}} \leq 9^{\tau d}$ and 
\begin{align} \label{eq:proof_excit_conc_temp1}
    \Big\|\frac{1}{T} \sum_{l=1}^m  (F_l \otimes N_l) - I_{\tau d}  \Big\|_2 \leq 2 \sup_{u \in \calN_{1/4}} \abs{\frac{1}{T} u^\top \Big(\sum_{l=1}^m F_l \otimes N_l \Big) u - 1}.
\end{align}
Now we write
\begin{align*}
    \frac{1}{T} u^\top \Big(\sum_{l=1}^m F_l \otimes N_l \Big) u = \frac{1}{T} \sum_{t=1}^T \sum_{l=1}^m \dotprod{\eta_{l,t}}{U f_l}^2 = \frac{1}{T} \sum_{t=1}^T \sum_{l=1}^m B_{l,t}(u)
\end{align*}
where $B_{l,t}(u) := \dotprod{\eta_{l,t}}{U f_l}^2$ are independent sub-exponential random variables. Moreover, we can bound the $\psi_1$ norm (see for e.g., \cite[Section 2.7, Lemma 2.7.6]{HDPbook}) of $B_{l,t}(u)$ as
\begin{equation*}
    \norm{B_{l,t}(u)}_{\psi_1} = \norm{\dotprod{\eta_{l,t}}{U f_l}}_{\psi_2}^2 \leq c_1 R^2 \norm{U f_l}_2^2 \leq c_1 R^2 \frac{\theta\tau}{m}.
\end{equation*}
Furthermore, $\expec[B_{l,t}(u)] = \norm{U f_l}_2^2$, and so, 
\begin{align*}
    \frac{1}{T} u^\top \Big(\sum_{l=1}^m F_l \otimes N_l \Big) u - 1 &= \frac{1}{T} \sum_{t=1}^T \sum_{l=1}^m \Big(\underbrace{B_{l,t}(u) - \expec[B_{l,t}(u)]}_{=: \tilde{B}_{l,t}(u)} \Big)
\end{align*}
where $\norm{\tilde{B}_{l,t}(u)}_{\psi_1} \leq c_2 \norm{B_{l,t}(u)}_{\psi_1} \leq c_3 R^2 \theta\tau/m$.

Then using Bernstein inequality for sums of independent centered sub-exponential random variables \cite[Theorem 2.8.2]{HDPbook}, we obtain for any $t > 0$
\begin{align*}
    \prob\Bigg(\frac{1}{T} \abs{\sum_{t=1}^T \sum_{l=1}^m \tilde{B}_{l,t}(u)} \geq t \Bigg) \leq 2 \exp\Big(-c_5 \min\set{\frac{t^2}{R^4 \theta^2 \tau^2}, \frac{t}{R^2 \theta \tau}} Tm \Big).
\end{align*}
Choosing $t = \varepsilon/2$ where $\varepsilon = R^2 \theta\tau \max(\delta,\delta^2)$ with $\delta > 0$, leads to 
\begin{equation*}
    \prob\Bigg(\frac{1}{T} \abs{\sum_{t=1}^T \sum_{l=1}^m \tilde{B}_{l,t}(u)} \geq \frac{\varepsilon}{2} \Bigg) \leq 2 \exp\Big(-c_5 \delta^2 Tm \Big).
\end{equation*}
Now choose $\delta = c_6 \sqrt{\frac{\theta \tau d}{T m}}$ for constant $c_6 > 0$ suitably large, this implies 
\begin{align*}
    \prob\Bigg(\max_{u \in \calN_{1/4}} \frac{1}{T} \abs{\sum_{t=1}^T \sum_{l=1}^m \tilde{B}_{l,t}(u)} \geq \frac{\varepsilon}{2} \Bigg) \leq 9^{\tau d} 2 \exp(-c_5 \delta^2 Tm) \leq  2 \exp(-c_7 \theta\tau d) 
\end{align*}
for some constant $c_7 > 0$. Plugging this in \eqref{eq:proof_excit_conc_temp1}, we finally obtain
\begin{align*}
    \prob \left(\Big\| \sum_{l=1}^m  (F_l \otimes N_l) - T I_{\tau d}  \Big\|_2 \leq C R^2 \max\set{\sqrt{\frac{\theta^3 \tau^3 d T}{ m}}, \frac{\theta^2 \tau^2 d}{m}} \right) \geq 1 - 2 \exp(-c \theta \tau d)
\end{align*}
for appropriate constants $C, c > 0$. The statement of the proposition now follows trivially.

%
\subsection{Proof of Proposition \ref{prop:e3_bd_imp_event_subsp}}
  Recall the definition of $f_l \in \matR^\tau$ from \eqref{eq:f_l_def}. We then begin by observing that
    \begin{align}
        \Vtau^\top Q^\top \eta &= \sum_{l=1}^m \sum_{t=1}^T f_l \otimes (x_{l,t} \otimes \eta_{l,t+1}), \nonumber \\
        \text{ and } \ \Vtau^\top Q^\top Q \Vtau &= \sum_{l=1}^m \sum_{t=1}^T \Big[(f_l \otimes x_{l,t}) (f_l \otimes x_{l,t})^\top \Big] \otimes I_d, \label{eq:vtau_qq_exp_subcons}
    \end{align}
    and so our goal is to bound
    \begin{align*}
        \Big\|\Big(\sum_{l=1}^m \sum_{t=1}^T \Big[(f_l \otimes x_{l,t}) (f_l \otimes x_{l,t})^\top \Big] \otimes I_d + (T/4)I_{\tau d^2}\Big)^{-1/2} \Big(\sum_{l=1}^m \sum_{t=1}^T f_l \otimes (x_{l,t} \otimes \eta_{l,t+1}) \Big)\Big\|_2^2.
    \end{align*}
We proceed now as in the proof of Propositon \ref{prop:bd_self_norm_laplsmooth}. To this end, with $(\calF_t)_{t \geq 1}$ a filtration as before, we have the following analogous result to \cite[Lemma 9]{abbasi11} for our setting.
\begin{lemma}
For any $\mu \in \matR^{\tau d^2}$, consider for any $t \geq 1$,
    \begin{align*}
        M_{t}^\mu := \exp\Bigg( \sum_{s=1}^t \sum_{l=1}^m\left[  \frac{\dotprod{\mu}{(f_l \otimes x_{l,s}) \otimes \eta_{l,s+1}}}{R} - \frac{\norm{\mu}_{(f_lf_l^\top) \otimes (x_{l,s}x_{l,s}^\top) \otimes I_d}^2}{2} \right] \Bigg).
    \end{align*}
    Let $\stoptime$ be a stopping time w.r.t $(\calF_{t})_{t=1}^\infty$. Then $M_{\stoptime}^\mu$ is well defined and $\expec[M_{\stoptime}^\mu] \leq 1$.    
\end{lemma}
\begin{proof}
Denoting
    \begin{equation*}
        D_{t}^\mu := \exp\Bigg(  \sum_{l=1}^m\left[  \frac{\dotprod{\mu}{f_l \otimes (x_{l,t} \otimes \eta_{l,t+1})}}{R} - \frac{\norm{\mu}_{f_lf_l^\top \otimes ((x_{l,t}x_{l,t}^\top) \otimes I_d)}^2}{2} \right] \Bigg), 
    \end{equation*}
    Observe that $D_{t}^\mu, M_{t}^\mu$ are both $\calF_{t+1}$-measurable. Moreover, it is easy to verify that $\expec[D_{t}^\mu \vert \calF_{t}] \leq 1$w which implies $(M_{t}^\mu)_{t=1}^{\infty}$ is a supermartingale, and also $\expec[M_{t}^\mu] \leq 1$. The arguments for showing that $M_{\stoptime}^\mu$ is well-defined, and $\expec[M_{\stoptime}^\mu] \leq 1$ is identical to \cite[Lemma 9]{abbasi11}, and hence omitted
\end{proof}
Denoting $\Mbar = \frac{T}{4} I_{\tau d^2}$ for convenience, we obtain in a completely analogous manner as in the proof of \cite[Theorem 1]{abbasi11}, that for any $\delta\in (0,1)$, it holds w.p at least $1-\delta$, 
\begin{align}
    \Big \|(\Vtau^\top Q^\top Q \Vtau +\Mbar)^{-1/2} (\Vtau^\top Q^\top \eta) \Big\|_2^2 
    &\leq 2R^2 \log \Bigg(\frac{\det(\Vtau^\top Q^\top Q \Vtau + \Mbar)^{1/2} \det(\Mbar)^{-1/2}}{\delta} \Bigg)  \nonumber \\
    &= 2R^2 \log \Bigg(\frac{\det(\frac{4 \Vtau^\top Q^\top Q \Vtau}{T} + I_{\tau d^2})^{1/2}}{\delta} \Bigg). \label{eq:selfnorm_excit_subcons_temp1}
\end{align}
It remains to upper (L\"owner) bound $\Vtau^\top Q^\top Q \Vtau$. Using the expression in \eqref{eq:vtau_qq_exp_subcons}, we see that
\begin{align*}
    \norm{\Vtau^\top Q^\top Q \Vtau}_2 = \norm{\sum_{l=1}^m \sum_{t=1}^T (f_l f_l^\top) \otimes (x_{l,t} x_{l,t})^\top }_2 \leq \left(\frac{\theta\tau}{m}\right) \sum_{l=1}^m \norm{\sum_{t=1}^Tx_{l,t} x_{l,t}^\top}_2.
\end{align*}
Then in an analogous manner to the proof of Lemma \ref{lem:lowner_bd_Umat}, one can readily verify that $\norm{\sum_{t=1}^Tx_{l,t} x_{l,t}^\top}_2 \leq  \gamma_3(T,m)/\delta$ w.p at least $1-\delta$, for $\gamma_3(T,m)$ as defined in \eqref{eq:gamma_conds_sub_theorem}. Using the subsequent bound $\Vtau^\top Q^\top Q \Vtau \preceq \frac{\theta \tau \gamma_3(T,m)}{m\delta} I_{\tau d^2}$ in \eqref{eq:selfnorm_excit_subcons_temp1} then leads to the statement of the proposition after some minor simplification.

%% file: appendix_useful_tools.tex
\section{Useful tools} \label{sec:useful_tools}
The following result was used implicitly in \cite[Section 9]{Sarkar19}, and is obtained from arguments used within the proof of \cite[Proposition 7.5]{Sarkar19}. For convenience, we outline the proof below.
\begin{proposition}[\cite{Sarkar19}] \label{prop:stable_mat_trace_bound}
    Let $A = P^{-1} \Lambda P$ be the Jordan Canonical Form of $A$ where $\Lambda = \blkdiag(J_{k_1}(\lambda_1),\dots,J_{k_s}(\lambda_s))$, with $J_k(\lambda) = \lambda I_k + N_k$ denoting a $k \times k$ Jordan block. Here $N_k$ is a nilpotent matrix with $1$'s on the first upper off-diagonal and $0$'s everywhere else. If $\rho(A) \leq 1 + \frac{c}{T}$ for some constant $c > 0$, and $T \geq 2$, then 
    \begin{align*}
        \Tr(\Gamma_T(A) - I) \leq c_1 e^{2c} \Big(\sum_{i=1}^s T^{2k_i - 1} \Big) \norm{P}_2^2 \norm{P^{-1}}_2^2
    \end{align*}
    for some constant $c_1 > 0$. If $T = 1$ then $\Tr(\Gamma_T(A) - I) \leq e^{2c} (\sum_{i=1}^s k_i^2) \norm{P}_2^2 \norm{P^{-1}}_2^2$
\end{proposition}
\begin{proof}
 We will follow the steps in  \cite[Proposition 7.5]{Sarkar19}. First note that
 \begin{equation*}
  \norm{A^t}_F^2  = \norm{P^{-1} \Lambda^t P}_F^2 \leq \norm{P^{-1}}_2^2 \norm{P}_2^2 \norm{\Lambda^t}_F^2,
 \end{equation*}
 which in turn implies
 \begin{equation*}
     \Tr(\Gamma_T(A) - I) = \sum_{t=1}^T \Tr(A^t (A^t)^{\top}) \leq \norm{P^{-1}}_2^2 \norm{P}_2^2 \sum_{t=1}^T \Tr(\Lambda^t (\Lambda^t)^H).
 \end{equation*}
Here ${(\Lambda^t)^H}$ denotes the Hermitian transpose of $\Lambda^t$ as $\Lambda$ may have complex-value entries. Since
\begin{equation*}
    \Lambda^t (\Lambda^t)^H = 
\begin{pmatrix}
    \sum_{t=1}^T J^t_{k_1}(\lambda_1) (J^t_{k_1}(\lambda_1))^H & & &\\
    & \ddots & & \\
    & & \ddots & \\
    & & &\sum_{t=1}^T J^t_{k_s}(\lambda_s) (J^t_{k_s}(\lambda_s))^H
\end{pmatrix},
\end{equation*}
it will suffice to bound $\Tr(\sum_{t=1}^T J^t_{k_1}(\lambda_1) (J^t_{k_1}(\lambda_1))^H)$. Now as mentioned in \cite[Proposition 7.5]{Sarkar19}, there is a $k_1 \times k_1$ unitary diagonal matrix $D$ such that $D (\lambda_1 I_{k_1} + N_{k_1}) (\lambda_1 I_{k_1} + N_{k_1})^H D^H$ is a real matrix, and in particular, 
\begin{equation*}
    D (\lambda_1 I_{k_1} + N_{k_1}) (\lambda_1 I_{k_1} + N_{k_1})^H D^H = (\abs{\lambda_1} I_{k_1} + N_{k_1}) (\abs{\lambda_1} I_{k_1} + N_{k_1})^H.
\end{equation*}
Thus we obtain
\begin{align*}
    \Tr\left(\sum_{t=1}^T J^t_{k_1}(\lambda_1) (J^t_{k_1}(\lambda_1))^H \right) 
    &= \Tr\left(D \left[\sum_{t=1}^T J^t_{k_1}(\lambda_1) (J^t_{k_1}(\lambda_1))^H \right] D^H \right) \\
    &= \Tr\left(\underbrace{\sum_{t=1}^T (\abs{\lambda_1} I_{k_1} + N_{k_1}) (\abs{\lambda_1} I_{k_1} + N_{k_1})^H}_{=: B} \right) \\
    &= \sum_{l=1}^{k_1} B_{ll}.
\end{align*}
If $\abs{\lambda_1} \leq 1+\frac{c}{T}$, then eq. (34) in the proof of  \cite[Proposition 7.5]{Sarkar19} implies
\begin{equation} \label{eq:temp_proof_1_stabsarkar}
     B_{ll} \leq e^{2c} T \sum_{j=0}^{k_1-l} T^{2j} = e^{2c} T \left(\frac{T^{2(k_1-l+1)}-1}{T^2-1}\right) \lesssim e^{2c} T^{2(k_1-l)+1},
\end{equation}
if $T \geq 2$. This in turn implies 
$$\sum_{l=1}^{k_1} B_{ll} \lesssim e^{2c} \sum_{l=1}^{k_1} T^{2(k_1-l)+1} \lesssim e^{2c} T^{2k_1-1}$$ which readily leads to the bound in the first part of the proposition. 

When $T = 1$, notice from the first inequality in  \eqref{eq:temp_proof_1_stabsarkar} that $B_{ll} \leq e^{2c} (k_1-l+1)$, and so, 
\begin{equation*}
   \sum_{l=1}^{k_1} B_{ll} \leq  e^{2c} \sum_{l=1}^{k_1} (k_1-l+1) = \sum_{i=1}^{k_1} i \leq e^{2c} k_1^2.
\end{equation*}
This leads to the stated bound in the second part of the proposition.
\end{proof}